\documentclass[smallextended,envcountsect]{svjour3}

\smartqed

\usepackage{graphicx}

\usepackage{latexsym,amsmath,amsfonts,amscd,amssymb,verbatim,hyperref}
\usepackage{color}
\usepackage{cite}
\usepackage[shortlabels]{enumitem}

\def\tto{\;{\lower 1pt \hbox{$\rightarrow$}}\kern -10pt
	\hbox{\raise 2pt \hbox{$\rightarrow$}}\;}

\def\epsilon{\varepsilon}

\begin{document}
	
	\title{Error Bounds for a Class of Cone-Convex Inclusion Problems}
	
	\author{Nguyen Quang Huy \and Nguyen Huy Hung \and Nguyen Van Tuyen \and Hoang Ngoc Tuan}
	
	\institute{Nguyen Quang Huy \at Department of Mathematics, Hanoi Pedagogical University 2, Vinh Phuc, Vietnam\\
		nqhuy@hpu2.edu.vn\and 
		Nguyen Huy Hung \at
		Department of Mathematics, Hanoi Pedagogical University 2, Vinh Phuc, Vietnam\\
		nguyenhuyhung@hpu2.edu.vn
		\and 
		Nguyen Van Tuyen \at
		Department of Mathematics, Hanoi Pedagogical University 2, Vinh Phuc, Vietnam\\
		nguyenvantuyen83@hpu2.edu.vn	
		\and 
		Hoang Ngoc Tuan, Correspoding author \at
		Department of Mathematics, Hanoi Pedagogical University 2, Vinh Phuc, Vietnam\\
		hoangngoctuan@hpu2.edu.vn			
	}
	
	\date{Received: date / Accepted: date}
	
	\titlerunning{Error Bounds for a Class of Cone-Convex Inclusion Problems}
	
	\authorrunning{N. Q. Huy, N. H. Hung, N. V. Tuyen, H. N. Tuan}
	
	\maketitle
	

\begin{abstract}  In this paper, we investigate error bounds for cone-convex inclusion problems in finite-dimensional settings of the form $f(x)\in K$, where $K$ is a smooth cone and  $f$ is a continuously differentiable and $K$-concave function. We show that local error bounds for the inclusion can be characterized by the Abadie constraint qualification around the reference point. In the case where $f$ is an affine function,  we precisely identify the conditions under which the inclusion admits global error bounds. Additionally, we derive some properties of smooth cones, as well as regular cones and strictly convex cones.
	
\end{abstract}


\section{Introduction}

Let $K$ be a non-empty, closed and convex cone in $\mathbb{R}^m$, and let $h: \mathbb{R}^n\to\mathbb{R}^m$ be a  differentiable  function that is $K$-concave, i.e., 
\begin{equation*}\label{C_concave}
h\big[(1-\lambda)x+\lambda y\big]-\big[(1-\lambda)h(x)+\lambda h(y)\big]\in K   \ \, \forall x,y\in\mathbb{R}^n,\ \, \forall \lambda\in[0,1].
\end{equation*}
The problem of finding points $x\in \mathbb{R}^n$ satisfying 
\begin{equation}\label{convex_cone_inclusion_pre}
h(x)\in K
\end{equation}
is called the cone-convex inclusion problem.
We denote the solution set of \eqref{convex_cone_inclusion_pre} by
\begin{equation}\label{sigma_define_pre}
S: =\{x \in \mathbb{R}^n\, : h(x)\in K \big\},
\end{equation}
which is assumed to be nonempty throughout the paper. 

The constraint set of numerous optimization problems can be formulated in the framework of \eqref{convex_cone_inclusion_pre}. For instance, when $K=\mathbb{R}^p_-\times \{0\}^{m-p}$, the inclusion reduces to a finite system of $p$ convex inequalities and $m-p$ affine equations.  

In  more general settings,  the problem and associated error bounds has been studied extensively in the literature, including works by Robinson \cite{Robinson_75,Robinson_76}, Ursescu \cite{Ursescu_75}, Ngai and Théra  \cite{Ngai_Thera_2005}, and Burke and Deng \cite{Burke_Deng_2009}.

The inclusion \eqref{convex_cone_inclusion_pre} is said to  satisfy a local error bound at a point $\bar x\in  S $ if there exists a neighborhood $U$ of $\bar x$ and a constant $\alpha>0$ such that
\begin{equation}\label{error_bound_definition} 
d(x, S )\le \alpha\, d(h(x),K)\, \forall x\in   U,
\end{equation}
where $d(x,D):=\inf\{\|x-y\|:\, y\in D\}$ for any subset $D$ of a Euclidean space. If $U=\mathbb{R}^n$, then \eqref{error_bound_definition}  is said to be a global error bound for \eqref{convex_cone_inclusion_pre}.

Error bounds have numerous applications in various areas of optimization. They play a crucial role in the convergence analysis of numerical algorithms for optimization problems and variational inequalities. Error bounds have also proven to be significant in the context of metric regularity and sensitivity analysis. The reader is referred to the paper \cite{Pang_1997} for a comprehensive survey of the broad theory and rich applications of error bounds in mathematical programming. 

Now, we present some constraint conditions for \eqref{convex_cone_inclusion_pre}, which are closely related to the existence of error bounds for \eqref{convex_cone_inclusion_pre}. Let $\hat S$ be a subset of $S$. According to \cite{Burke_Deng_2009}, we say that the inclusion \eqref{convex_cone_inclusion_pre} satisfies the Abadie constraint qualification  for $\hat S $ (ACQ($\hat S$) for brevity)  if 	
\begin{equation}\label{abadie_general_case_pre}
h'(x)^*[N_{K}(h( x))]=N_{ S }(x)\, \forall x\in  \hat S,
\end{equation}
where $N_D(x)$ denotes the normal cone of a convex set $D$ at $x\in D$ in a Eulidean space.
The inclusion satisfies the Mangasarian–Fromovitz constraint qualification  for $\hat S$ (MFCQ($\hat S$) for brevity) if $${\rm Ker}h'(x)^*=N_K(h( x))\, \forall x \in \hat S,$$ where ${\rm Ker}\,T$ and $T^*$ denote the kernel and the adjoint operator of the linear operator $T$ between Euclidean spaces, respectively. Moreover, \eqref{convex_cone_inclusion_pre} is said to be satisfied the Slater constraint qualification (SCQ for brevity) if there is a point $\hat{x}\in \mathbb{R}^n$ such that $h(\hat{x})\in {\rm ri}K$, where ${\rm ri} K$ denotes the relative interior of $K$. From \cite[Proposition 2]{Burke_Deng_2009} we have
\begin{equation*}\label{CQ_relations}
{\rm MFCQ}(\{x\})\Rightarrow {\rm SCQ}\Rightarrow {\rm ACQ}(S)
\end{equation*}
for any $x\in S$.
In particular, if  ${\rm int}K$ is nonempty, then $ {\rm SCQ}\Rightarrow{\rm MFCQ}(S)$.

Next, we discuss some literature concerning error bounds for \eqref{convex_cone_inclusion_pre} and their relationship with the above constraint qualifications. 
When $K=\mathbb{R}^m_-$, the inclusion \eqref{convex_cone_inclusion_pre} is reduced to the following system of finitely many differentiable convex inequalities
\begin{equation}\label{polyhedral_cone_inclusion}
h_i(x)\le 0,\ \, i=1,\ldots,m,
\end{equation}
where each component $h_i$ of $h$ is a convex function on $\mathbb{R}^n$.  (In fact, by \cite[Theorem 25.5]{Rocka}, we obtain that $h_i$, for $i=1,\ldots,m,$ is continuously differentiable. Therefore, in this case, the function $h$ is also continuously differentiable.)   It was Hoffman \cite{Hoffman_52} who first established  global error bounds  for \eqref{polyhedral_cone_inclusion} in the case where each $h_i$ is an affine function. If each $h_i$ is a convex quadratic function, the equivalence between a global error bound for \eqref{polyhedral_cone_inclusion} and  ACQ($S$) holds (see \cite[Theorem 4.2]{Li_1997}). In more general cases, global error bounds for \eqref{polyhedral_cone_inclusion} can be obtained under the Slater constraint qualification and an asymptotic qualification condition \cite{Magasarian_1985}. Without these two qualification conditions, it is proven in \cite[Theorem 3.5]{Li_1997} that a local error bound for \eqref{polyhedral_cone_inclusion} at $\bar x\in S$ is equivalent to ACQ($S\cap U$), where $U$ is a neighborhood of $\bar x$. 

The equivalence between local error bounds and the Abadie constraint qualification for \eqref{convex_cone_inclusion_pre} still holds when $K$ is an arbitrary polyhedral convex cone (see \cite[Theorem 3]{Burke_Deng_2009}). However, when $K$ is not polyhedral, this equivalence generally does not hold. In fact, we have the following result.

\begin{theorem}\label{error_bound_equivalence}{\rm (cf. \cite[Theorem 2]{Burke_Deng_2009})}
	Assume that $K\subset \mathbb{R}^m$ is a non-empty, closed and convex cone and $h:\mathbb{R}^n\to\mathbb{R}^m$ is a continuously differentiable and $K$-concave function. Consider the following condition:
	\begin{equation}\label{general_inclusion_pre}
	h'(x)^*\big[N_{K}(h(x))\big]\cap \bar B_{\mathbb{R}^n}\subset \tau h'(x)^*\big[N_{K}(h(x))\cap \bar B_{\mathbb{R}^m}\big].
	\end{equation}
	Then, the inclusion \eqref{convex_cone_inclusion_pre} satisfies a local error bound at $\bar x\in  S $ if and only if the exist $\delta>0$ and  $\tau>0$ such that both the ACQ and \eqref{general_inclusion_pre} hold for all $x\in  S \cap \bar B(\bar x,\delta)$. Moreover, a global error bound holds for \eqref{convex_cone_inclusion_pre}  if and only if the exists $\tau>0$ such that both the ACQ and \eqref{general_inclusion_pre} hold for all $x\in  S $.
\end{theorem}

It is implied from the above theorem that, roughly speaking, the ACQ is a necessary condition for the existence of error bounds for \eqref{convex_cone_inclusion_pre}. It is also worth noting that the SCQ and MFCQ are sufficient conditions for the existence of a local error bound for \eqref{convex_cone_inclusion_pre}  at every point in $S$ \cite[Theorem 4]{Burke_Deng_2009}.

When $h$ is an affine function of the form $h(x)=Ax+b$, where $A: \mathbb{R}^{n}\to  \mathbb{R}^m$ is linear and $b\in \mathbb{R}^m$, the inclusion \eqref{convex_cone_inclusion_pre} simplifies to $Ax+b \in K$. The existence of global error bounds for this inclusion has been studied in \cite{Burke_Tseng_1996,Burke_Deng_2009,Ng_Yang_2002}. Specifically, for the case where $K$ is a second-order cone,  the authors in \cite[Section 5]{Ng_Yang_2002} identify exactly when the inclusion has a global and when it does not.

In \cite{Roshchina_Tucel_2019}, the authors introduce the concept of \textit{smooth cones}, a subclass of \textit{strictly convex cones}. (The definitions of these types of cones will be presented in the next section.)  Smooth cones include \textit{second-order cones}, which are among the most significant non-polyhedral cones.  Second-order cones play a crucial role in second order-cone programming (SOCP), which has numerous applications in optimization \cite{Alizadeh_Goldfarb_2003,LVBL_1997}. Notably, the feasible set of a SOCP problem with a single cone constraint can be formulated as  \eqref{convex_cone_inclusion_pre} when $K$ is a second-order cone and $h$ is an affine function \cite{LVBL_1997}. 
Two generalizations of second-order cones, $p$-\textit{order cones} ($p\in (1,+\infty)$) and \textit{circular cones}, are also included in the class of smooth cones. Their properties and applications have been investigated in various studies \cite{Ito_Luoenco_2017,Xue_Ye_2000,Zho_Chen_2013,Zho_Chen_Mordukhovich_2014}. 
These observations illustrate the wide range of potential applications for smooth cones, emphasizing their significance for further study.

The aim of this paper is to investigate error bounds for \eqref{convex_cone_inclusion_pre} under the assumptions that the cone $K$ is smooth and the function $h$ is continuously differentiable and $K$-concave. Our contributions can be outlined as follows:

\medskip
$\bullet$ We first present some properties of regular cones, strictly convex cones and smooth cones that will be  useful for our purposes. In particular, we demonstrate that every smooth cone is strictly convex and provide several characterizations of smooth cones.

\medskip
$\bullet$ We establish the equivalence between a local error bound for \eqref{convex_cone_inclusion_pre} at $\bar x\in S$ and the ACQ($S\cap U$),  where $U$ is some neighborhood of $\bar x$. Additionally, we present an example demonstrating that global error bounds for \eqref{convex_cone_inclusion_pre} cannot be characterized by the ACQ. 

\medskip
$\bullet$  We precisely determine the cases in which the inclusion $Ax+b\in K$ has a global error bound. These results extend those in \cite[Section 5]{Ng_Yang_2002}, where $K$ is assumed to be a second-order cone.

\medskip

The rest of the paper is organized as follows. In Section 2, we present some notations and preliminaries. A number of results involving regular cones, strictly convex cones and smooth cones are  stated in Section 3. Section 4 focuses on local error bounds for \eqref{convex_cone_inclusion_pre}. Several auxiliary results are also presented in this section. In Section 5, we address  global error bounds for \eqref{convex_cone_inclusion_pre} when $h$ is affine.  Some concluding remarks are given in the last section.

\section{Preliminaries}

In this paper, the spaces $\mathbb{R}^n$ and $\mathbb{R}^m$ are referred to Euclidean spaces with dimension $n$ and $m$, respectively. In the following, we will present some notations for a general (finite-dimensional) Euclidean space $X$. 

The inner product and norm in $X$ are defined as follows: for any $x,y\in X$, $\langle x,y\rangle=x^Ty$, and $\|x\|=(x^Tx)^{1/2}$. Open and closed balls with center $x$ and radius $\varepsilon$ are denoted by $B(x, \varepsilon)$ and $\bar B(x, \varepsilon)$, respectively. Specifically, the open and closed unit balls in $ X$ are represented as $B_{ X}$ and $\bar B_{ X}$, respectively.   If $C$ is a subset of $ X$, we denote its \textit{interior},\textit{ relative interior}, \textit{closure} and \textit{boundary}  by ${\rm int}C$, ${\rm ri}C$, ${\rm cl}C$ (or $\bar C$) and $ {\rm bd} C $, respectively. For a subspace $L$ of $ X$, its \textit{orthogonal complement} is defined as $L^\perp := \{y\in X : x^T y = 0, \forall x \in L\}$. 

Given a set $D\subset X$, 
we denote by ${\rm span}(D)$ the intersection of all linear subspaces of $X$ that contain
$D$.

Consider  a nonempty convex set $C\subset X$. By the dimension of $C$, we mean the dimension of its affine hull. The \textit{normal cone} $N_C(\bar x)$ to $C$ at $\bar x\in C$ is defined as 
$$N_C(\bar x):=\{y\in  X:\, \langle y,x-\bar x\rangle \le0\, \forall x\in C\}.$$
According to \cite{Rocka}, a \textit{supporting half-space} to $C$ is a closed half-space which contains $C$ with a point in its boundary. A  \textit{supporting hyperplane} to $C$ is a hyperplane forming the boundary of a  supporting half-space to $C$. 
A convex set $\mathcal{F} \subset C$ is said to be a \textit{face} of $C$ if the following  condition holds: if $x, y \in C$ and $\lambda x + (1 - \lambda)y \in \mathcal{F}$ for some $0 < \lambda < 1$, then $x, y \in \mathcal{F}$ .
A face $\mathcal{F}$ of $C$ is said to be \textit{proper} if $\mathcal{F} \ne C$.  An \textit{exposed face} of $C$ (aside from C itself and possibly  the empty set) is a set of the form $C\cap H$, where $H$ 
is a non-trivial supporting hyperplane to $C$ \cite[p. 162]{Rocka}. 

A set  $C\subset  X$ is called a \textit{cone} if for any $x\in C$ and $\lambda\ge0$ we have $\lambda x\in C$. Every face of a closed convex cone is also a closed convex cone. In fact, a face of a closed convex cone $C$ is a closed convex cone $\mathcal{F}\subset C$ satisfying that if $x\in C$, $y-x\in C$ and $y\in \mathcal{F}$, then $x\in\mathcal{F}$ (see \cite[Definition 1.3]{Barker_1981} and \cite[Definition 2.3]{Barker_1973}). The \textit{polar cone} of a convex cone $C$ is defined as 
$$C^\circ:=\{y\in  X: \langle x, y\rangle\le0\, \forall x\in C\}.$$
If $C$ is closed then $(C^\circ)^\circ=C$ (see \cite[Theorem 14.1]{Rocka}).  A set $\ell\subset C$ is called a \textit{ray} if it has the form $\ell=\{\lambda x: \lambda \ge0\}$, where $x \ne 0$. 

A ray of a convex cone $C$ is called an \textit{extreme ray} if it is a face of $C$ (see \cite[p. 162]{Rocka}). A ray which is an exposed face is called an \textit{exposed ray}. Every exposed ray is an extreme ray (see \cite[p. 163]{Rocka}).

We say that a cone $C$ is pointed if $C\cap (-C)=\{0\}$.  As defined in \cite{Roshchina_Tucel_2019}, a cone is regular if it is a pointed, closed and convex set with nonempty interior. 

\begin{definition}\label{strictlyconvex_smooth_definition}
	{\rm (i) (cf. \cite[pp. 328--329]{Ito_Luoenco_2017})
		A regular cone $C$ is called \textit{strictly convex} if every face of $C$,  except $C$ itself and $\{0\}$, has dimension one;
		
		(ii) (cf. \cite[p. 2372]{Roshchina_Tucel_2019}) A regular cone $C$ is said to be \textit{smooth}  if every boundary point of $C$ is on an
		extreme ray of $C$ and the normal cone of $C$ at every non-zero point of an arbitrary extreme ray of $C$ has dimension one.}
\end{definition}

It is well-known that the $p$-\textit{cone},  defined as
\begin{equation}\label{p_cones}
\mathcal{K}_p:=\Big\{x=(t,u)^T\in\mathbb{R}\times\mathbb{R}^{m-1}:t\ge \|u\|_p \Big\},
\end{equation}
is strictly convex \cite{Ito_Luoenco_2017}. 
Here,  $p\in (1,+\infty)$ and $\|\cdot\|_p$ denotes the $p$-norm in $\mathbb{R}^{m-1}$ defined as
$\|u\|_p=\displaystyle\big(\sum_{i=1}^{m-1}|u_i|^p\big)^{\frac{1}{p}}$ for all $u=(u_1,\cdots,u_{m-1})^T\in\mathbb{R}^{m-1}$.
The class of $p$-cones contains  one of the most important non-polyhedral cones, called the \textit{second-order cone}: 
\begin{equation}\label{2_cones}
\mathcal{K}_2=\Big\{x=(t,u)^T\in\mathbb{R}\times\mathbb{R}^{m-1}:t\ge \|u\|\Big\}.
\end{equation}
The second-order cone is also included in the class of \textit{circular-cones}, which are defined as
\begin{equation}\label{circular_cones}
\mathcal{L}_\theta:=\Big\{x=(t,u)^T\in\mathbb{R}\times\mathbb{R}^{m-1}:t\ge \cos\theta \|x\|\Big\},
\end{equation}
where $\theta\in (0,\pi/2)$ \cite{Zho_Chen_2013,Zho_Chen_Mordukhovich_2014}.

In the next section, we will see that  every smooth cone is strictly convex. Furthermore, we will show in detail that $p$-cones and circular cones, including the second-order cone, are smooth.

Consider a linear operator $T:\, X\to Y$ between Euclidean spaces $X$ and $Y$. The adjoint operator of $T$, $T^*:\, Y\to X$, is defined as 
$$\langle Tx,y\rangle=\langle x,T^*y\rangle\,\forall \, x\in X,\,y\in Y.$$
The \textit{image} and the \textit{kernel} of $T$ are denoted by ${\rm Im}\,T$ and ${Ker}\,T$, respectively. Then we have $({Ker}\,T^*)^\perp= {\rm Im}\,T$. For a linear subspace $L$ of $X$, denote by $T|_L$ the restriction of $T$ to $L$. 

For a multifunction $F: X\rightrightarrows Y$ between Euclidean spaces $X$ and $Y$, we denote by ${\rm dom} F$ its \textit{domain}. Following \cite[pp. 25--26]{Bank_et_al_1982}, the multifunction is called \textit{lower semicontinuous} (l.s.c.) at  a point $\bar x\in {\rm dom} F$ if for any open subset $V\subset Y$ satisfying $F(\bar x)\cap V\ne \emptyset$, there exists $\varepsilon>0$ such that $F(x)\cap V\ne \emptyset$ for all $x\in B_X(\bar x,\varepsilon)$. Moreover, it is called \textit{strongly lower semicontinuous} (strongly l.s.c.) at  a point $\bar x\in {\rm dom} F$ if for each $y\in F(\bar x)$ there exist $\varepsilon>0$ and $\delta>0$ such that $B_Y(y,\varepsilon)\subset F(x)$ for every $x\in B_X(\bar x,\delta)$. Clearly, if $F$ is strongly l.s.c. at $\bar x$, then it is l.s.c. at the same point.

\section{Properties of regular cones, strictly convex cones and smooth cones}

In this section, we present several properties of regular cones, strictly convex cones and smooth cones that are crucial for establishing error bounds of the inclusion \eqref{convex_cone_inclusion_pre}. While some of the results may be known, we provide detailed proofs for the sake of completeness. We begin with the following properties of regular cones. 
\begin{proposition} \label{regular_cone}
	Let $C\subset X$ be a regular cone. Then, the following hold:
	
	(i) The polar cone $C^\circ$ is regular and if $\bar x\in  {\rm bd} C\setminus\{0\}$, then 
	\begin{equation}\label{normal_cone_form}
	N_C(\bar x)=\{y\in{\rm bd} C^\circ:\, \langle y,\bar x\rangle=0\}. 
	\end{equation}
	
	(ii) If $x$ and $y$ are two non-zero vectors contained in a same ray of $ {\rm bd}  C$, then we have $N_C(x)=N_C(y)$. 
\end{proposition}
\begin{proof}
	(i) Clearly, $C^\circ$ is closed and convex. The fact that $C^\circ$ is pointed and has non-empty interior is deduced from \cite[Propostition 1.1.15]{Auslender_Teboulle_2003}, along with the observation that $(C^\circ)^\circ=C$. Hence, $C^\circ$ is regular.
	
	Now, assume that $\bar x\in  {\rm bd}  C\setminus\{0\}$. Taking an arbitrary vector $y\in N_C(\bar x)$, by the definition of the normal cone, we have 
	$$
	\langle y,x-\bar x\rangle \le0\, \forall x\in C.
	$$
	Letting $x:=2\bar x\in C$, we get $\langle y,\bar x\rangle \le 0$; letting $x:=0\in C$, we get $\langle y,\bar x\rangle \ge 0$. Thus, $\langle y,\bar x\rangle =0$. It follows that
	$$ \langle y,x\rangle\le \langle y,\bar x\rangle=0\, \forall x\in C.$$
	This yields $y\in C^\circ$. If $y\in {\rm int}\, C^\circ$, then invoking  \cite[Proposition 1.1.15]{Auslender_Teboulle_2003} again, we get $\langle y,\bar x\rangle<0$, a contradiction. So, $y \in {\rm bd}  C^\circ$, implying \eqref{normal_cone_form}.

	(ii) Assuming that two non-zero vectors $x$ and $y$ are contained in the same ray of $ {\rm bd}  C$, there exists $\lambda>0$ such that $y=\lambda x$. Taking an arbitrary vector $z\in N_C(x)$, from the proof of part (i) we have $\langle z,x\rangle=0$ and $\langle z,u\rangle\le0$ for every $u\in C$. This implies that $\langle z,y\rangle=\langle z,\lambda x\rangle=\lambda\langle z, x\rangle=0$. Hence, $\langle z,u-y\rangle \le0$ for every $u\in C$. This means that $z\in N_C(y)$. Thus, $N_C(x)\subset N_C(y)$. Using similar arguments as above, we also have $N_C(y)\subset N_C(x)$. So, $N_C(x)= N_C(y)$. $\hfill\Box$
\end{proof}

Next, we consider regular cones in relation to linear subspaces.
\begin{proposition}\label{equality_of_subspaces}
	Let $L_1$ and $L_2$ be linear subspaces and  $C$ be a regular cone in $X$. Assume that  $L_1\cap {\rm int}C\ne \emptyset$. Then, the following assertions hold:
	
	(i)   $(u + L_1) \cap {\rm int} C \neq \emptyset$ for all $u \in X$;
	
	(ii)	If $L_1\cap C=L_2\cap C$, then $L_1=L_2$.
\end{proposition}
\begin{proof} 
	(i) Given an arbitrary vector $\bar{u} \in X$, if 
	\begin{equation}\label{nonempty_interior_cap}
	(\bar u+L_1)\cap {\rm int} C= \emptyset,
	\end{equation}
	then there exists a hyperplane that separates $\bar{u} + L_1$ and $C$. This hyperplane is characterized by a non-zero vector $w \in \mathbb{R}^m$ and a scalar value $\alpha$, such that 	
	\begin{equation}\label{separation_condition_L_C}
	\left\langle w, \bar{u} + u \right\rangle \leq \alpha\le \langle w, v \rangle \quad \forall u \in L_1\, , v\in C.
	\end{equation}
	Since $C$ is a cone, it follows from the second inequality of \eqref{separation_condition_L_C} that $\langle w,v\rangle\ge0$ for all $v\in C$. Furthermore, as $L_1$ is a linear space, it follows from the first inequality of \eqref{separation_condition_L_C} that $\left\langle w, u \right\rangle = 0$ for all $u \in L_1$. 
	
	On the other hand, by the assumption, there exists a vector $\bar v \in  L_1\cap {\rm int}C$. Hence, $\langle w,\bar v\rangle=0$ and $B(\bar v,\varepsilon)\subset C$ for some $\varepsilon>0$. Since $w\ne0$, $B(\bar v,\varepsilon)$ is full-dimensional and $\langle w,v\rangle\ge0$ for all $v\in C$, there is a vector $\bar z\in B(\bar v,\varepsilon)$ satisfying $\langle w,\bar z\rangle>0$. This yields $\langle w,2\bar v-\bar z\rangle<0$. However, $2\bar v-\bar z\in B(\bar v,\varepsilon)\subset C$, as $\|	(2\bar v-\bar z)-\bar v\|=\|\bar v-\bar z\|<\varepsilon$.	This contradicts the fact that $\langle w,v\rangle\ge0$ for all $v\in C$.

	(ii) 	Take any vector $u\in L_1$. If $u=0$ then $u\in L_2$. Now, consider the case $u\ne0$. Since $L_1\cap {\rm int}C\ne \emptyset$, there is a vector $\bar u\in L_1\cap{\rm int}C$. Thus, $\bar u\in L_1$ and there exists $\varepsilon>0$ such that $B(\bar u,\varepsilon)\subset {\rm int}C$. Set $\tilde{u}=\bar u+\varepsilon u/(2\|u\|)$. We have $\tilde{u}\in B(\bar u,\varepsilon)$, as $\|\tilde{u}-\bar u\|<\varepsilon$, and $\tilde{u}\in L_1$, as $L_1$ is linear. Hence, 
	$$\tilde{u}\in L_1\cap B(\bar u,\varepsilon)\subset L_1\cap C=L_2\cap C\subset L_2. $$
	We also have
	$$\bar{u}\in  L_1\cap C=L_2\cap C\subset L_2. $$
	Accordingly, $\varepsilon u/(2\|u\|)=\tilde{u}-\bar u\in L_2$. This yields $u\in L_2$, as $L_2$ is linear. Therefore, $L_1\subset L_2$. 
	
	To prove the reverse inclusion $L_2\subset L_1$, we   observe that $L_2\cap {\rm int}C\ne \emptyset$. Indeed, if $L_2\cap {\rm int}C = \emptyset$ then $L_2\cap C\subset  {\rm bd}  C$. Thus, $L_1\cap C=L_2\cap C\subset  {\rm bd}  C$, implying $L_1\cap {\rm int}C = \emptyset$, which is a contradiction. Therefore, using similar arguments as above, we obtain $L_2\subset L_1$. $\hfill\Box$
\end{proof}

\begin{proposition}\label{interrior_perpendicular_cone}
	Let $C$ be a regular cone and $L$ be a linear subspace in $X$. Then, we have the following:
	
	(i) $L\cap {\rm int}C^\circ\ne \emptyset$ if and only if $L^\perp \cap C=\{0\}$;
	
	(ii) $\{0\}\ne L\cap C^\circ \subset  {\rm bd}  C^\circ$ if and only if $\{0\}\ne L^\perp \cap C\subset  {\rm bd}  C$.
	
\end{proposition}
\begin{proof}
	(i) First, assume that $L\cap \,{\rm int}C^\circ= \emptyset$ but $L^\perp \cap C\ne\{0\}$. Then, it follows from the latter that there exists a non-zero vector $x\in L^\perp \cap C$. Take any vector $y\in L\cap \,{\rm int}C^\circ$. Since $x\in C\setminus \{0\}$ and $y\in {\rm int}C^\circ$, by \cite[Propostition 1.1.15]{Auslender_Teboulle_2003} we derive $\langle x, y\rangle<0$. However, because $x\in L^\perp$ and $y\in L$, we have $\langle x, y\rangle=0$, leading to a contradiction.
	
	Now, assume that $L^\perp \cap C=\{0\}$ but $L\cap {\rm int}C^\circ\ne \emptyset$. Then,  there exists a non-zero vector $\bar z\in X$ and a real number $\alpha$ satisfying the following two conditions
	$$\langle \bar z, x\rangle\ge\alpha\, \forall x\in L$$
	and 
	$$\langle \bar z, u\rangle\le\alpha\, \forall u\in C^\circ.$$
	Combining the first condition with the fact that $L$ is linear and the second one with the fact that $C^\circ$ is a cone, we deduce that
	$$\langle \bar z, x\rangle=0\, \forall x\in L$$
	and
	$$\langle \bar z, u\rangle\le0\, \forall u\in C^\circ.$$
	Thus, $\bar z\in L^\perp$ and $\bar z\in (C^\circ)^\circ=C$, implying $L^\perp \cap C\ne\{0\}$, which is a contradiction.

	(ii) This is a direct consequence of assertion (i). $\hfill\Box$
\end{proof}

Now we discuss some properties of   strictly convex cones. The next result provides significant characterizations of strictly convex cones.

\begin{proposition} \label{strictly_convex_cone_0}
	Let $C\subset X$ be a regular cone. Then, the following statements hold:
	
	(i) The cone $C$ is strictly convex if and only if  every ray lying on the boundary of $C$ is an extreme ray;
	
	(ii)  The cone $C$ is  strictly convex if and only if  for every $x,y\in C\setminus\{0\}$, with $y\ne \lambda x$ for all $\lambda>0$, we have
	\begin{equation}\label{strictly_Convex_interior}
	(x,y):=\{(1-\lambda) x+\lambda y:\, 0<\lambda<1\}\subset {\rm int}C.
	\end{equation}

\end{proposition}
\begin{proof}
	(i) Suppose that $C$ is strictly convex. Let $\ell$ be any ray  on $ {\rm bd}  C$. Since $\ell\cap {\rm int}C=\emptyset$, there is a hyperplane $H\subset X$ separating $\ell$ and $C$. Clearly, $\ell\subset C\cap H$ and $C\cap H\ne C$. This implies that $C\cap H$ is a proper exposed face of $C$. Combining this with the fact that $C$ is strictly convex yields  ${\rm dim} \,(C\cap H)=1$, implying $\ell=C\cap H$. It follows that $\ell$ is an exposed ray, hence an extreme ray.
	
	Now, suppose that every ray lying on the boundary of $C$ is an extreme ray. Let $\mathcal{F}\subset C$ be any proper face of $C$ such that $\mathcal{F}\ne \{0\}$. By \cite[Lemma 2.6]{Barker_1973}, we have $\mathcal{F}\subset  {\rm bd}  C$. Since $\mathcal{F}$ is convex, there exists a vector $x\in {\rm ri}\,\mathcal{F}$. Clearly, $0\ne x\in  {\rm bd}  C$. Hence, the ray $S$ containing $x$ lies on  $ {\rm bd}  C$ and  is therefore an extreme ray, implying that $S$ is a face. Moreover, observe that $x\in {\rm ri}S$. Invoking \cite[Corollary 18.1.2]{Rocka} we obtain $\mathcal{F}=S$. This means that ${\rm dim}\mathcal{F}=1$. So, $C$ is strictly convex.
	
	(ii)  Suppose that $C$ is strictly convex. Take any non-zero vectors $x,y\in C$, with $y\ne \lambda x$ for all $\lambda>0$. If  $x\in {\rm int}C$ or $y\in {\rm int}C$, then by \cite[Theorem 6.1]{Rocka}, we get \eqref{strictly_Convex_interior}. If $x,y\in  {\rm bd}  C\setminus\{0\}$ but \eqref{strictly_Convex_interior} does not hold then there exists a vector $u\in  {\rm bd}  C\cap(x,y)$. The ray $S\subset  {\rm bd}  C$ containing $u$ is a face and has a nonempty intersection with $(x,y)$. It follows that $x,y\in S$, contradicting the condition that  $y\ne \lambda x$ for all $\lambda>0$. Thus, \eqref{strictly_Convex_interior} holds.
	
	Conversely, suppose that   for every non-zero vectors $x,y\in C$, with $y\ne \lambda x$ for all $\lambda>0$,  \eqref{strictly_Convex_interior} holds. Let $\mathcal{F}\subset C$ be any non-zero and proper face of $C$. Invoking again \cite[Lemma 2.6]{Barker_1973}, we get $\mathcal{F}\cap  {\rm int}  C=\emptyset$. If $\mathcal{F}$ is not a ray, then there are two non-zero vectors $x,y\in \mathcal{F}$ such that $y\ne \lambda x$ for all $\lambda>0$. Since $\mathcal{F}\subset C$, it follows from \eqref{strictly_Convex_interior} that $(x,y)\subset {\rm int}C$. Moreover, as $\mathcal{F}$ is convex, we have $(x,y)\subset \mathcal{F}$. This implies that  $(x,y)\subset \mathcal{F}\cap{\rm int}C$, which is a contradiction. Hence, $\mathcal{F}$ is a ray, implying ${\rm dim}(\mathcal{F})=1$. So, $C$ is strictly convex. $\hfill\Box$
\end{proof}

\begin{remark}\label{circular_cone_sc} The circular cone $\mathcal{L}_\theta$ defined by \eqref{circular_cones} is strictly convex. 	
	Indeed, according to \cite[Theorem 2.1]{Zho_Chen_2013}, there exists a positive definite matrix $M$ depending on $\theta$ such that  $\mathcal{L}_\theta=M^{-1}(\mathcal{K}_2)$ and $\mathcal{K}_2=M(\mathcal{L}_\theta)$, where $\mathcal{K}_2$ is the second-order cone defined by \eqref{2_cones}. Thus, taking any $x,y\in \mathcal{L}_\theta\setminus\{0\}$ with $y\ne \lambda x$ for all $\lambda>0$, we have $Mx,My\in \mathcal{K}_2\setminus\{0\}$ and $My\ne \gamma Mx$ for all $\gamma>0$. Since $\mathcal{K}_2$ is strictly convex, invoking Proposition \ref{strictly_convex_cone_0}(ii), we obtain $(Mx,My)\subset {\rm int}\mathcal{K}_2$. It follows that
	$$(x,y)=M^{-1}[(Mx,My)]\subset M^{-1}({\rm int}\mathcal{K}_2)={\rm int}[M^{-1}(\mathcal{K}_2)]={\rm int}\mathcal{L}_\theta.$$
	So, $\mathcal{L}_\theta$ is strictly convex.
\end{remark}

\begin{proposition} \label{strictly_convex_cone}
	Let $C\subset X$ be a strictly convex cone. Then, the following statements hold:
	
	(i) For every linear subspace $L\subset  X$ satisfying $\{0\}\ne L\cap C \subset  {\rm bd}  C$, the set  $L\cap C$ is a ray;
	
	(ii) If $x$ and $y$  are two non-zero vectors in $ {\rm bd}  C$ satisfying $N_C(x) = N_C(y)$, then $x$ and $y$ are contained in the same ray.
\end{proposition}
\begin{proof}

	(i) If $L\cap C$ is not a ray, then there exist non-zero vectors  $x,y\in L\cap C$ with $y\ne \lambda x$ for all $\lambda>0$. By \eqref{strictly_Convex_interior}, $(x,y)\subset {\rm int}C$, impyling that $L\cap {\rm int}C\ne\emptyset$. This contradicts the assumption.

	(ii) Suppose that $x$ and $y$  are two non-zero vectors in $ {\rm bd}  C$ satisfying the condition $N_C(x) = N_C(y)$. If $x$ and $y$ are not contained in the same ray, then by \eqref{strictly_Convex_interior}, we have $(x,y)\subset {\rm int}C$. In particular, $(x+y)/2\in {\rm int}C$. On the other hand, taking any $z\in N_C(x)$, we get $\langle z,x\rangle=\langle z,y\rangle=0$. It follows that $\langle z,(x+y)/2\rangle=0$. Hence, by applying \cite[Propostition 1.1.15]{Auslender_Teboulle_2003}, we obtain $(x+y)/2\in {\rm bd}C$, which is a  contradiction.  $\hfill\Box$
\end{proof}

\begin{proposition} \label{strictly_convex_cone_1}
	Assume that the cone $C\subset X$ is regular and that its polar cone $C^0$ is  strictly convex. If $x$ is a non-zero vector in $ {\rm bd}  C$, then $N_C(x)$ is a ray. $\hfill\Box$

\end{proposition}
\begin{proof}
	Since $C$ is regular, invoking \eqref{normal_cone_form}, we derive $N_C(x)\subset {\rm bd}C^\circ$. If $N_C(x)$ is not a ray, then there exist non-zero vectors  $u,v\in N_C(x)$ with $v\ne \lambda u$ for all $\lambda>0$. Since $C^\circ$ is strictly convex, by \eqref{strictly_Convex_interior}, we get $(u,v)\subset {\rm int}C^\circ$. Since $N_C(x)$ is convex, we have $(u,v)\subset N_C(x)$. Therefore, $N_C(x)\cap {\rm int}C^\circ\ne\emptyset$, contradicting the fact that $N_C(x)\subset {\rm bd}C^\circ$. So, $N_C(x)$ is a ray.  $\hfill\Box$
\end{proof}

The rest of this section focuses on studying properties of smooth cones. An important characterization of smooth cones is stated as follows.

\begin{proposition} \label{smooth_cone_then_strictly_convex}
	A cone $C$ in a Euclidean space $X$ is  smooth if and only if it is strictly convex and the normal cone of $C$ at every point of its boundary (except the origin) has dimension one.
\end{proposition}
\begin{proof} Assume that the cone $C$ is smooth. Let $\ell$ be any ray on the boundary of $C$. Then, there exists a non-zero vector $x\in {\rm bd}C$ such that $\ell=\{\lambda\,x\,: \lambda\ge0\}$. Since $C$ is smooth and $\ell$ is the only ray containing $x$, it follows that $\ell$ is an extreme ray. By invoking Proposition \ref{strictly_convex_cone_0}(i), we deduce that $C$ is strictly convex. 
	
	Now, let $x$ be any non-zero vector on the boundary of $C$. Since $C$ is strictly convex, invoking  Proposition \ref{strictly_convex_cone_0}(i) again, we conclude that the ray $\ell=:\{\lambda\,x\,: \lambda\ge0\}$ is an extreme ray. Thus, $N_C(x)$ has dimension one due to the fact that the cone $C$ is smooth.
	
	The reverse implication is straightforward, and its proof is therefore omitted.   $\hfill\Box$
\end{proof}

\begin{remark} It follows directly from the above proposition that every smooth cone is strictly convex.
\end{remark}

A further relationship between smooth cones and strictly convex cones is given as follows.

\begin{proposition} \label{smooth_cone_then_strictly_convex_1}
	A cone $C$ in a Euclidean space $X$ is  smooth if and only if both  $C$ and $C^\circ$ are strictly convex.
\end{proposition}
\begin{proof} Assume that the cone $C$ is smooth. 	Clearly, $C$ is strictly convex. Next, we will show that $C^\circ$ is strictly convex.
	
	Since $C$ is regular, by Proposition \ref{regular_cone}(i), we have $C^\circ$ is regular.
	
	We now show that every ray in $ {\rm bd}  C^\circ$ is  a face of $C^\circ$. Let $\ell\subset  {\rm bd}  C^\circ$ be a ray. Then, there exists a vector $\bar x \in  {\rm bd}  C^\circ\setminus \{0\}$ such that $\ell=\{\lambda \bar x:\, \lambda\ge0\}$. It follows from \cite[Lemma 2.8]{Barker_1973} that the set $$\Phi(\bar    x):=\{x\in C^\circ:\, \exists \lambda>0,\bar x-\lambda x\in C\}$$
	is a face of $C$ (generated by $\bar x$). It is clear that $\ell\subset \Phi(\bar x)$.
	
	Moreover, for any $x\in \Phi(\bar x)$, there exists $\lambda>0$ such that $\bar x-\lambda x\in C^\circ$. Since $\bar x\in  {\rm bd}  C^\circ$, we get $N_{C^\circ}(\bar x)\ne \{0\}$. Hence, there exists a vector $\bar y\in N_{C^\circ}(\bar x)\setminus\{0\}$, and  we have 
	\begin{equation*}
	\langle \bar y, (\bar x-\lambda x)-\bar x\rangle=-\lambda\langle \bar y,x\rangle\le 0. 
	\end{equation*}
	This yields $\langle \bar y,x\rangle\ge 0$. Thus, noting that $x\in C^0$, we derive 
	$$\langle x,y-\bar y\rangle=\langle x,y\rangle-\langle x,\bar y\rangle\le 0\quad \forall y\in C.$$
	This implies that $x\in N_C(\bar y)$, yielding $\Phi(\bar x)\subset N_C(\bar y)$.
	Additionally, since $C$ is smooth, $N_C(\bar y)$ is a ray. Combining this with the fact that $\ell\subset \Phi(\bar x)$, we can deduce that  $\ell=\Phi(\bar x)(=N_C(\bar y))$. So, $\ell$ is a face. 
	
	Next, let $\mathcal{F}\subset C^\circ$ be any non-zero and proper face of $C^\circ$. By \cite[Lemma 2.6]{Barker_1973}, we have $\mathcal{F}\subset  {\rm bd}  C^\circ$. Since $\mathcal{F}$ is convex, there exists a vector $x\in {\rm ri}\,\mathcal{F}$. Clearly, $x\in  {\rm bd}C\setminus \{0\}$. Therefore, the ray $\ell_x:=\{\lambda x:\,\lambda \ge0\}$ lies on  $ {\rm bd} C^\circ$. It follows that $\ell_x$ is a face of $C^\circ$.  Furthermore, since $x\in {\rm ri}\ell_x$, by \cite[Corollary 18.1.2]{Rocka}, we obtain $\mathcal{F}=\ell_x$. This means that ${\rm dim}\mathcal{F}=1$. So, $C^\circ$ is strictly convex.
	
	Now, assume that  $C$ and $C^\circ$ are strictly convex. For any non-zero vector $x\in {\rm bd}C$, by Proposition \ref{strictly_convex_cone_1}, we get the normal cone $N_{C}(x)$ is a ray. This means that it has dimension one. Consequently, it follows from Proposition \ref{smooth_cone_then_strictly_convex} that $C$ is smooth. 
	
	The proof is complete. $\hfill\Box$
\end{proof}

The following result is a direct consequence of Proposition \ref{smooth_cone_then_strictly_convex_1}.
\begin{corollary} \label{smooth_cone_and_polar}
	Assume that $C$ is a smooth cone in a Euclidean space $X$. Then, its polar cone $C^\circ$ is also smooth. 
\end{corollary}

To conclude this section, we will prove the claim made in the previous section.
\begin{corollary} \label{smooth_cone_example}
	Both the $p$-cone defined in \eqref{p_cones} and the circular cone  defined in \eqref{circular_cones} are smooth.
\end{corollary}
\begin{proof} Note that  $(\mathcal{K}_p)^\circ=-\mathcal{K}_q$, where $q$ is a positive number satisfying $\frac{1}{p}+\frac{1}{q}=2$. Hence, both $\mathcal{K}_p$ and $(\mathcal{K}_p)^\circ$ are strictly convex. By applying Proposition \ref{smooth_cone_then_strictly_convex_1}, we obtain $\mathcal{K}_p$ is smooth.
	
	It follows from \cite[Theorem 2.1(c)]{Zho_Chen_2013} that $(\mathcal{L}_\theta)^\circ=-\mathcal{L}_{\frac{\pi}{2}-\theta}$. Thus, both $\mathcal{L}_\theta$ and $(\mathcal{L}_\theta)^\circ$ are strictly convex, due to  Remark \ref{circular_cone_sc}. Applying  Proposition \ref{smooth_cone_then_strictly_convex_1} again, we get $\mathcal{L}_\theta$ is smooth. $\hfill\Box$
\end{proof}
\section{Local Error Bounds for Smooth Cone-Convex Inclusions}
In this section, we assume that $K$ is a smooth cone in $\mathbb{R}^m$ and that $h: \mathbb{R}^n\to\mathbb{R}^m$ is a continuously differentiable and $K$-concave function. Recall that $S$ is assumed to be the non-empty solution set of \eqref{convex_cone_inclusion_pre} defined in \eqref{sigma_define_pre}. We will establish the equivalence between local error bounds and the ACQ around the reference point for \eqref{convex_cone_inclusion_pre}. Before stating and proving the main result of this section, we need several auxiliary results.

\begin{lemma}\label{Ker_normal_cone} Under the above assumptions about the cone $K$ and the function $h$, we have
	\begin{equation*}\label{Kernel_intersection_cone}
	{\rm Ker}h'(x)^*\cap K^\circ={\rm Ker}h'(u)^*\cap K^\circ\,
	\end{equation*} 
	for every $ x,u\in  S $ satisfying $h(x)=h(u)=0$.
\end{lemma}
\begin{proof} Fix any  $x\in S$ and $u\in S$ satisfying $h(x)=h(u)=0$. 
	By the symmetric role of $x$ and $u$, we need only prove that 
	\begin{equation}\label{Kernel_intersection_cone_inclusion}
	{\rm Ker}\,h'(x)^*\cap K^\circ\subset{\rm Ker}\,h'(u)^*\cap K^\circ.
	\end{equation}  
	Take any vector $w\in {\rm Ker}\,h'(x)^*\cap K^\circ$. If $w=0$ then $w\in {\rm Ker}\,h'(u)^*\cap K^\circ$. If $w\ne0$,  then 
	$\bar w:=w/||w||\in{\rm Ker}\,h'(x)^*\cap K^\circ.$ Hence, $h'(x)^* \bar w=0$ and $\bar w\in K^\circ  \cap\bar  B_{\mathbb{R}^m}$. 
	
	Now, utilizing the arguments in the proof of \cite[Proposition 1]{Burke_Deng_2009}, we consider the function $\varphi (z)=\langle \bar w,h(z)\rangle$ for every $z\in \mathbb{R}^n$. It follows from \cite[Lemma 2]{Burke_Deng_2009} that $\varphi$ is convex. Moreover, $\varphi'(x)=h'(x)^*\bar w=0$. This implies that $x$ is a global solution of the problem $\displaystyle\min\{\varphi(z):\,z\in\mathbb{R}^n\}$. If there exists a point $\hat z\in S$ satisfying $\varphi(\hat z)>\varphi(x)$, then, by noting that $\bar w\in \bar B_{\mathbb{R}^m}\cap K^\circ$, we obtain
	\begin{align*}
	d(h(\hat z),K) &=\sup\big\{\langle v,h(\hat z)\rangle-{\rm supp}_K(v):\, \|v\|\le1\big\} \\
	&\ge \langle \bar w,h(\hat z)\rangle -{\rm supp}_K(\bar w)\\
	&> \langle \bar w,h(x)\rangle -{\rm supp}_K(\bar w)\\
	&=d(h(x),K) \\
	&=0,
	\end{align*}
	where ${\rm supp}_K(\cdot)$ denotes the support function of $K$, the first equation follows from \cite[Theorem A.1, Part 2]{Burke_Deng_2002}, and the second equation follows from \cite[Theorem A.1, Part 5]{Burke_Deng_2002}. This contradicts the fact that $\hat z\in S$. Thus, $S$ belongs to the solution set of the problem $\displaystyle\min\{\varphi(z):\,z\in\mathbb{R}^n\}$. This implies that $h'(y)^*\bar w=0$ for every $y\in S $. In particular, $h'(u)^*\bar w=0$, yielding $h'(u)^* w=0$. Therefore, $w\in {\rm Ker}\,h'(u)^*\cap K^\circ$. So, the inclusion \eqref{Kernel_intersection_cone_inclusion} holds. 	 $\hfill\Box$
\end{proof}

\begin{lemma}\label{lsc_of_multifunction}
	Assume that $T:X\to  Y$ is a linear mapping between Euclidean spaces  and that $C\subset X$ is a strictly convex cone satisfying ${\rm Ker}T\cap {\rm int}C\ne\emptyset$. Consider a multifunction $F: {\rm Im}T\rightrightarrows  X $ defined as 
	$$F(v)= \{u\in  X :\, Tu=v \}\cap C$$
	for every $v\in {\rm Im}T$. Then, $F$ is l.s.c. at every non-zero point $ v\in {\rm dom}F$.
	
	Moreover, if $T(C)$ is closed, then there exists a positive number $\tau$ such that
	\begin{equation}\label{cone_image_inclusion}
	T(C)\cap \bar B_{ Y}\subset \tau\, T(C\cap \bar B_X ).
	\end{equation}
\end{lemma}
\begin{proof} Fix an arbitrary non-zero vector $\bar v \in {\rm dom}F$.
	Setting
	\begin{align*}
	\Gamma_1:\,\, &{\rm Im}T\hspace{0.3cm}\rightrightarrows \hspace{0.9cm} X \\
	&\hspace{0.3cm} v\hspace{0.5cm}\mapsto \{u\in  X :\, Tu=v \},
	\end{align*}
	\begin{align*}
	\Gamma_2:\,\, &{\rm Im}T\hspace{0.2cm}\rightrightarrows \hspace{0.3cm} X \\
	&\hspace{0.3cm} v\hspace{0.4cm}\mapsto \hspace{0.1cm} {\rm int}C,
	\end{align*}
	and
	\begin{align*}
	\Gamma_3:\,\, &{\rm Im}T\hspace{0.2cm}\rightrightarrows \hspace{0.3cm} X \\
	&\hspace{0.3cm} v\hspace{0.4cm}\mapsto \hspace{0.2cm} C,
	\end{align*}
	we will verify that $\Gamma_1,\Gamma_2$ and $\Gamma_3$ satisfy the conditions in \cite[Corollary 2.2.5.1]{Bank_et_al_1982}.
	
	Clearly, $F=\Gamma_1\cap\Gamma_3$. From \cite[Corollary 2.1]{Robinson_73} we can deduce that $\Gamma_1$ is l.s.c. at $\bar v$. Since ${\rm int}C$ is open and nonempty, $\Gamma_2$ is strongly l.s.c. at $\bar v$. Moreover, $\Gamma_2( v)\subset \Gamma_3( v)$ for every $v\in {\rm Im}T$.  Therefore, the conditions (1)--(3) as in \cite[Corollary 2.2.5.1]{Bank_et_al_1982} are satisfied. 
	
	Next, we will show that	
	$$\Gamma_1(\bar v)\cap\Gamma_3(\bar v)\subset {\rm cl}\big( \Gamma_1(\bar v)\cap\Gamma_2(\bar v) \big) $$
	or, equivalently, 
	\begin{equation}\label{closure_inclusion}
	\{u\in  X :\, Tu=\bar v \}\cap C\subset {\rm cl}\big[\{u\in  X :\, Tu=\bar v \}\cap {\rm int}C\big].
	\end{equation}
	Since $F(\bar v)\ne \emptyset$, there exists $\bar u\in C$ such that $T \bar u=\bar v$. Hence, we deduce that 
	$ \{u\in  X :\, Tu=\bar v \}=\bar u+{\rm Ker} T.$
	Thus, inclusion \eqref{closure_inclusion} becomes 
	\begin{equation}\label{closure_inclusion_2}
	(\bar u+{\rm Ker} T)\cap C\subset {\rm cl}\big[(\bar u+{\rm Ker} T)\cap {\rm int} C\big].
	\end{equation}
	
	Now, we are going to prove \eqref{closure_inclusion_2}.
	Since ${\rm Ker} T\cap {\rm int}C\ne\emptyset$, by Proposition \ref{equality_of_subspaces}(i), we get 
	$(\bar u+{\rm Ker} T)\cap {\rm int} C\ne \emptyset.$
	Hence, there exists $\bar z\in (\bar u+{\rm Ker} T)\cap {\rm int} C$. Fix any  $z\in(\bar u+{\rm Ker} T)\cap C$.   Invoking \eqref{strictly_Convex_interior}, we obtain the interval $(\bar z,z)$ is contained in ${\rm int}C$. Clearly, $(\bar z,z)\subset \bar u+{\rm Ker} T$. Thus, $(\bar z,z)\subset (\bar u+{\rm Ker} T)\cap {\rm int} C$. From this we deduce that the sequence $\{\frac{1}{k}\bar z+\frac{k-1}{k} z\}_{k\ge1}$ is contained in $(\bar u+{\rm Ker} T)\cap {\rm int} C$ and is convergent to $z$. Therefore, \eqref{closure_inclusion_2} is proved. 
	
	Now suppose that \eqref{cone_image_inclusion} does not hold. Then, for each $k=1,2,\ldots$, there exists $v^k\in Y$ such that
	\begin{equation}\label{v_k_inclusion_2}
	v^k\in T(C)\cap \bar B_{Y} 
	\end{equation}
	and
	\begin{equation}\label{v_k_not_inclusion_2}
	v^k\notin  k \,T\big[C\cap \bar B_{X}\big].
	\end{equation}
	
	It follows from \eqref{v_k_inclusion_2} and \eqref{v_k_not_inclusion_2} that $\|v^k\|\le 1$ and $v^k\ne0$. If $\|v^k\|<1$, then by replacing $v^k$ with the new $v^k:=v^k/\|v^k\|$, this new $v^k$ also satisfies \eqref{v_k_inclusion_2} and \eqref{v_k_not_inclusion_2}. Consequently, we can choose $\{v^k\}$ so that $\|v^k\|=1$ for every $k=1,2,\ldots$. Since the sequence $\{v^k\}$ is bounded, by extracting a subsequence if necessary, we can assume that $\lim v^k= \bar v$. Clearly, $\bar v\ne 0$. Moreover, since $T(C)$ is closed, it follows from \eqref{v_k_inclusion_2} that $\bar v\in T(C)\subset {\rm Im}T$. Hence, $F$ is l.s.c. at $\bar v$. By invoking \cite[Proposition 5.11]{Rocka_Wets_1998}, we deduce that the function $f: {\rm Im}T\to \mathbb{R}$ defined as $f(v)=d(0,F(v))$ for every $v\in {\rm Im}T$ is upper semicontinuous at $\bar v$, i.e.,  
	for any $\varepsilon>0$, the exists $\delta>0$ such that
	$$ d(0,F(v))<d(0,F(\bar v))+\varepsilon\, \forall v\in B(\bar v,\delta).$$
	It follows that $d(0,F(v^k))<d(0,F(\bar v))+\varepsilon$ for all sufficiently large $k$. This yields $F(v^k)\ne\emptyset$ for all sufficiently large $k$.  Since $F(\bar v)$ and $F(v^k)$ are closed, convex and nonempty for all sufficiently large $k$, there are $\bar z \in F(\bar v)$ such that $\|\bar z\|=d(0,F(\bar v))$ and $z^k\in F(v^k)$ such that $\|z^k\|=d(0,F( v^k))$. Therefore, $\|z^k\|<\|\bar z\|+\varepsilon$ for all sufficiently large $k$. This implies that
	$$v^k=T(z^k)\in (\|\bar z\|+\varepsilon)\,T\big(C\cap \bar B_{ X }\big)$$
	for all sufficiently large $k$, contradicting  \eqref{v_k_not_inclusion_2}. So, \eqref{cone_image_inclusion} holds.
	
	The proof is complete.  $\hfill\Box$
\end{proof}

\begin{lemma}\label{image_of_pointed_cone}
	Let $C\subset  X $ be a strictly convex cone and $T: X \to Y$ be a linear mapping. Assume that $\{0\}\ne{\rm Ker}T\cap C\subset  {\rm bd}  C$. Then $T(C)$ is a pointed cone in $ Y$.    	
\end{lemma}
\begin{proof} It is implied from the assumption that ${\rm Ker T}\cap \,{\rm int}C=\emptyset$. Hence,  there exists a hyperplane $H\subset  X $ separating ${\rm Ker T}$ and $C$. In fact, $C$ is contained in a closed half-space defined by $H$ and ${\rm Ker}T\subset H$. Combining this with the assumption implies that $\{0\}\ne H\cap C\subset  {\rm bd}  C$. So, applying  Proposition \ref{strictly_convex_cone}(i), we get $H\cap C={\rm Ker}T\cap C$ is a ray lying in $ {\rm bd}  C$. 
	
	Now, taking any non-zero vector $y\in T(C)$, there is a non-zero vector $x\in C$ such that $y=Tx$. If $x\in H$ then $x\in H\cap C={\rm Ker}T\cap C$, implying that $x\in {\rm Ker}T$, which is a contradiction. Thus, $x\notin H$. It follows that $u:=-x\notin H$. Noting that $C$ is pointed, we obtain $u\notin C$. Therefore, $(u+H)\cap C=\emptyset$, yielding $(u+{\rm Ker}T)\cap C=\emptyset$. This implies that $Tu=T(-x)=-y\notin T(C)$. So, $T(C)$ is pointed.  $\hfill\Box$
\end{proof}

\begin{lemma}\label{Span_Sigma}
	Let $D$ be a nonempty set in a Euclidean space $X$ and  $x$ and $y$ be arbitrary vectors in $ D $. Then, we have $${\rm span}( D -x)={\rm span}( D -y).$$
	
\end{lemma}
\begin{proof} We can assume that $${\rm span}( D -x)={\rm span}\{y-x,z^1-x,\ldots,z^p-x\},$$
	where $z^i\in  D , i=1,\ldots, p$. Moreover,
	$${\rm span}\{y-x,z^1-x,\ldots,z^p-x\}={\rm span}\{x-y,z^1-y,\ldots,z^p-y\}\subset {\rm span}( D -y).$$
	Hence, ${\rm span}( D -x)\subset{\rm span}( D -y).$ Since the roles of $x$ and $y$ are equivalent, we also have ${\rm span}( D -y)\subset{\rm span}( D -x).$ 
 $\hfill\Box$
\end{proof}

\begin{lemma}\label{Boundedness_of_norm_of_coefficients}
	Assume that the system $\{\bar z_1, \dots, \bar z_d\}$ is a basis of a Euclidean space $X$. Then, there exists a number $\varepsilon>0$ such that $\{ z_1, \dots,  z_d\}$ is also a basis of  $X$ for every $z_i\in \bar B(\bar z_i,\varepsilon) (i=1,\ldots,d) $. Moreover, the set
	$$\big\{(\alpha_1,\dots,\alpha_d)^T\in \mathbb{R}^d: \|\sum_{i=1}^{d}\alpha_i z_i\|\le 1,\,z_i\in \bar B(\bar z_i,\varepsilon) (i=1,\ldots,d) \big\}$$
	is bounded.
\end{lemma}
\begin{proof}
	Consider the function $f(z_1,\ldots,z_d)$ defined as the determinant of the matrix formed by $d$ vectors $z_1,\ldots,z_d$ in $X$. We have that $f$ is continuous and that $f(\bar z_1,\ldots,\bar z_d)\ne 0$. So, there exists $\varepsilon>0$ satisfying $f(z_1,\ldots,z_d)\ne 0$ for every $z_i\in \bar B(\bar z_i,\varepsilon) (i=1,\ldots,d) $. It follows that $\{ z_1, \dots,  z_d\}$ is a basis of  $X$ for every $z_i\in \bar B(\bar z_i,\varepsilon) (i=1,\ldots,d) $.

	Put $H=\mathbb{R}^d\times (\underbrace{X\times\cdots\times X}_\text{$d$ products})$, $\alpha =(\alpha_1,\ldots,\alpha_d)^T\in \mathbb{R}^d$, and consider the function
	\begin{align*}
	\varphi:\,&\hspace{1cm}H\hspace{1cm}\to\hspace{1.5   cm}\mathbb{R}\\
	&(\alpha,z_1,\ldots,z_d)\hspace{0.2   cm}\mapsto \|\alpha_1 z_1+\cdots+\alpha_n z_d\|
	\end{align*}
	Then, $\varphi$ is continuous in $H$ and $\varphi(\alpha,z_1,\ldots,z_d)\ne 0$ for every 
	$$(\alpha,z_1,\ldots,z_d)\in D:= {\rm bd}  \big(\bar B_{\mathbb{R}^d}(0,1)\big)\times \bar B_X(\bar z_1,\varepsilon)\times\cdots\bar B_X(\bar z_d,\varepsilon),$$
	since on this set $\alpha\ne0$ and $\{ z_1, \dots,  z_n\}$ is a basis of  $\mathbb{R}^n$.
	Hence, noting that $D$ is compact in $H$, $\varphi$ attains its minimum $m_0>0$ on $D$. Moreover, if $\|\alpha\|\ne0$ then $$\varphi(\alpha,z_1,\ldots,z_n)=\|\alpha\| \big\|\frac{\alpha_1}{\|\alpha\|}z_1+\cdots+\frac{\alpha_n}{\|\alpha\|}z_n\big\|\ge \|\alpha\|m_0.$$
	So, the requirement $\varphi(\alpha,z_1,\ldots,z_n)=\|\sum_{i=1}^{n}\alpha_i z_i\|\le 1$ implies that $\|\alpha\|m_0\le 1$ or $\|\alpha\|\le 1/m_0$.
 $\hfill\Box$
\end{proof}

\begin{lemma}\label{Boundary_point_of_set_isomorphism}
	Let  $T: X\to Y$ be a linear bijection between Euclidean spaces and $D$ be a subset of $X$ with nonempty interior. Then, for every $x\in  {\rm bd}  D$ we have $Tx\in  {\rm bd}  \big(T(D)\big)$.
\end{lemma}
\begin{proof}
	The proof is straightforward and is therefore omitted. $\hfill\Box$
\end{proof}

Now, we are ready to state our main result of this section.

\begin{theorem}\label{nonlinear_case_theorem} Assume that $K$ is a smooth cone in  $\mathbb{R}^m$ and $h:\mathbb{R}^n\to\mathbb{R}^m$ is a continuously differentiable and $K$-concave function. Then, the inclusion \eqref{convex_cone_inclusion_pre} has a local error bound at a vector $\bar x\in S $ if and only if there exists  a neighborhood $U$ of $\bar x$ such that the ${\rm ACQ}( S \cap U$) holds for \eqref{convex_cone_inclusion_pre}.
\end{theorem}

\begin{proof}   We need only to prove that  if there exists  a neighborhood $U$ of $\bar x$ such that the ${\rm ACQ}( S \cap U$) holds for \eqref{convex_cone_inclusion_pre}, then the inclusion has a local error bound at $\bar x $. Moreover, by Theorem \ref{error_bound_equivalence},    the inclusion \eqref{convex_cone_inclusion_pre}  has a local error bound at $\bar x$ if and only if there exist $\delta>0$ and  $\tau>0$  such that the ACQ($ S \cap \bar B(\bar x,\delta)$) holds, i.e., 
	\begin{equation}\label{abadie_general_case}
	h'(x)^*[N_{K}(h(x))]=N_{ S }(x)\,
	\end{equation}
	for every $x\in  S \cap \bar B(\bar x,\delta)$,  and
	\begin{equation}\label{general_inclusion}
	h'(x)^*\big[N_{K}(h(x))\big]\cap\bar B_{\mathbb{R}^n}\subset \tau h'(x)^*\big[N_{K}(h(x))\cap\bar B_{\mathbb{R}^m}\big]
	\end{equation}
	for every $x\in  S \cap \bar B(\bar x,\delta)$. 
	Hence, to prove the theorem, it is sufficient to show that the ${\rm ACQ}( S \cap U$) implies  \eqref{abadie_general_case} and \eqref{general_inclusion}.
	We divide the proof into two cases. 
	
	\noindent {\bf Case 1.} $h(\bar x)\neq 0$.
	
	Since ${\rm ACQ}( S \cap U$) holds for \eqref{convex_cone_inclusion_pre} with some neighborhood $U$ of $\bar x$, there exists $\delta>0$ such that \eqref{abadie_general_case} is satisfied for every $x\in  S \cap \bar B(\bar x,\delta)$. 
	
	Suppose, on the contrary, that \eqref{convex_cone_inclusion_pre} does not have a local error bound at $\bar x$. Then, there is no $\tau>0$ such that \eqref{general_inclusion} holds for all $x\in  S \cap \bar B(\bar x,\delta)$. Thus, for every $k=1,2,\ldots$, there exist $x^k\in  S \cap \bar B(\bar x,1/k)$ and $v^k\in \mathbb{R}^n$  satisfying 
	\begin{equation}\label{v_k_belong}
	v^k\in h'(x^k)^*\big[N_{K}(h(x^k))\big]\cap \bar B_{\mathbb{R}^n}
	\end{equation}
	and
	\begin{equation}\label{v_k_not_belong}
	v^k\notin  k h'(x^k)^*\big[N_{K}(h(x^k))\cap \bar B_{\mathbb{R}^m}].
	\end{equation}
	From \eqref{v_k_belong} and \eqref{v_k_not_belong}, and by using similar arguments in the proof of Lemma \ref{lsc_of_multifunction}, we can assume that  $\|v^k\|=1$ for every $k=1,2,\ldots$, and that $\displaystyle\lim_{k\to\infty} v^k= \bar v$. Then, we have $\|\bar v\|=1$. 
	
	It is easy to see that  $\displaystyle\lim_{k\to\infty} x^k= \bar x$. In addition, combining \eqref{abadie_general_case} and \eqref{v_k_belong}, we get $v^k\in N_{S}(x^k)$.  Hence, by invoking \cite[Proposition 6.6]{Rocka_Wets_1998}, we obtain
	\begin{equation}\label{bar_v_inclusion-2}
	\bar v\in N_{S}(\bar x)= h'(\bar x)^*[N_{K}(h(\bar x)].
	\end{equation}
	
	Since \eqref{convex_cone_inclusion_pre} has not a local error bound at $\bar x$, the MFCQ($\{\bar x\}$) does not hold, i.e., ${\rm Ker}\,h'(\bar x)^*\cap N_{K}(h(\bar x))\ne \{0\}$. Moreover, it follows from the facts that $K$ is smooth and  $h(\bar x)\ne 0$ that the normal cone $N_{K}(h(\bar x))$ is a ray. Thus, we derive the inclusion ${\rm Ker}\,h'(\bar x)^*\supset N_{K}(h(\bar x))$, yielding $\,h'(\bar x)^*[N_{K}(h(\bar x))]=\{0\}$. From this and \eqref{bar_v_inclusion-2} we get $\bar v =0$, contradicting the fact that $\|\bar v\|=1$. So,  there exists $\tau>0$ satisfying \eqref{general_inclusion} holds for all $x\in  S \cap \bar B(\bar x,\delta)$, implying that \eqref{convex_cone_inclusion_pre} has a local error bound. 
	
	\noindent {\bf Case 2.} $h(\bar x)= 0$. 
	
	In this case, we have $N_{K}(h(\bar x))=N_K(0)=K^\circ$. According to the first assertion of \cite[Lemma 3]{Burke_Deng_2009}, we get
	\begin{equation*}\label{inclusion_function_extra}
	h( u)+h'( u)(x- u)-h(x)\in K
	\end{equation*}
	for every $u\in S$ and $x\in\mathbb{R}^n$. 
	It follows that 
	\begin{equation}\label{image_equation}
	h'( u)(x- u)=h(x)+w(x),
	\end{equation}
	where $w(x)\in K$, for every $u\in S$ satisfying $h(u)=0$ and $x\in\mathbb{R}^n$.
	
	Next, note that if the inclusion \eqref{convex_cone_inclusion_pre} satisfies the SCQ,  then it has a local error bound at $\bar x$. Hence, we arrive at the conclusion.
	
	We now assume that  \eqref{convex_cone_inclusion_pre} does not satisfy the SCQ.
	Then, the MFCQ($\{\bar x\})$ does not hold, i.e., ${\rm Ker}\,h'(\bar x)^*\cap K^\circ\ne \{0\}$. Consequently, one of the following  two subcases must occur: either ${\rm Ker}\,h'(\bar x)^*\cap {\rm int}K^\circ\ne \emptyset$ or  $\{0\}\ne{\rm Ker}\,h'(\bar x)^*\cap K^\circ\subset  {\rm bd}  K^\circ$. 
	
	\noindent {\bf Subcase 2.1.} ${\rm Ker}\,h'(\bar x)^*\cap {\rm int}K^\circ\ne \emptyset.$
	
	Since the ${\rm ACQ}( S \cap U$) holds for \eqref{convex_cone_inclusion_pre} with some neighborhood $U$ of $\bar x$, there exists $\delta_0>0$ such that \eqref{abadie_general_case} holds for every $x\in  S \cap \bar B(\bar x,\delta_0)$.

	Now, we will show that 
	\begin{equation}\label{Sigma_subcase_1}
	S =\{x\in \mathbb{R}^n: h(x)=0\}.
	\end{equation}
	By the facts that  ${\rm Ker}\,h'(\bar x)^*\cap {\rm int}K^\circ\ne \emptyset$ and $({\rm Ker}\,h'(\bar x)^*)^\perp={\rm Im }\,h'(\bar x)$, applying Lemma \ref{interrior_perpendicular_cone}(i), we obtain 
	\begin{equation}\label{Ker_image_perpencular}
	{\rm Im }\,h'(\bar x)\cap K=\{0\}.
	\end{equation} 
	Noting that $h(\bar x)=0$, that $h(x)\in K$ if $x\in  S $ and that $K$ is a convex cone, it follows from \eqref{image_equation} that $h'(\bar x)(x-\bar x)=h(x)+w(x)\in K$ for every $x\in  S $. Combining this with \eqref{Ker_image_perpencular}, we derive 
	\begin{equation}\label{image_equal_0}
	h'(\bar x)(x-\bar x)=0
	\end{equation}
	for every $x\in S$. 
	Hence, $h(x)+w(x)=0$. This yields $h(x)=w(x)=0$, as $K$ is pointed. Therefore, we get \eqref{Sigma_subcase_1}.
	
	From \eqref{Sigma_subcase_1} and Lemma \ref{Ker_normal_cone} we obtain 
	\begin{equation*}\label{kernel_intersection_cone_2}
	{\rm Ker}\,h'(x)^*\cap K^\circ={\rm Ker}\,h'(\bar x)^*\cap K^\circ
	\end{equation*}
	for every $x\in  S $. Combining this with the assumption ${\rm Ker}\,h'(\bar x)^*\cap {\rm int}K^\circ\ne \emptyset$ yields 
	\begin{equation*}\label{kernel_intersection_interrior}
	{\rm Ker}\,h'( x)^*\cap {\rm int}K^\circ\ne \emptyset
	\end{equation*}
	for every $x\in  S $. 
	Thus, invoking  of Proposition \ref{equality_of_subspaces}(ii), we get 
	\begin{equation}\label{equality_of_kernel}
	{\rm Ker}\,h'( x)^*={\rm Ker}\,h'(\bar x)^*
	\end{equation}
	for every $x\in  S $.
	
	Using \eqref{Sigma_subcase_1} and \eqref{equality_of_kernel} and repeating the same arguments as in \eqref{Ker_image_perpencular}-\eqref{image_equal_0}, we also get 
	\begin{equation*}\label{image_equal_x}
	h'( x)(y-x)=0.
	\end{equation*}
	for every $x,\,y\in  S $. It follows that $y-x\in {\rm Ker}\, h'(x)$ for every $x,\,y\in  S $. This yields  
	\begin{equation*}\label{Sigma_Ker_inclusion}
	S -x\subset {\rm Ker}\, h'(x)
	\end{equation*} 
	for every $x\in S $.  Combining this with the fact that ${\rm Ker}\, h'(x)$ is a subspace of $\mathbb{R}^n$, we deduce that 
	$${\rm span }\,( S -x)\subset {\rm Ker}\, h'(x),$$
	for every $x\in S $. By Lemma \ref{Span_Sigma}, the subspace $L:=   {\rm span }\,( S -x)$ is constant for all $x\in S $. Moreover, we have
	$$L\subset  {\rm Ker}\, h'(x)=({\rm Im}\,h'(x)^*)^\perp$$
	for every $x\in S $. Hence
	\begin{equation}\label{Image_sigma_normal_cone_inclusion_1}
	{\rm Im}\,h'(x)^*\subset L^\perp,
	\end{equation}
	for every $x\in S $.
	On the other hand, since \eqref{abadie_general_case} holds for every $x\in  S \cap\bar B(\bar x,\delta_0)$, we infer that
	\begin{equation*}\label{Image_sigma_normal_cone_inclusion_2}
	{\rm Im}h'(x)^*\supset h'(x)^*(K^\circ)=N_{ S }(x)=N_{ S -x}(0)\supset N_L(0)=L^\perp\, 
	\end{equation*}
	for every $x\in  S \cap\bar B(\bar x,\delta_0)$.
	Combining this with \eqref{Image_sigma_normal_cone_inclusion_1}, we obtain
	\begin{equation}\label{Image_sigma_normal_cone_equation}
	{\rm Im}\,h'(x)^*= h'(x)^*(K^\circ)=L^\perp
	\end{equation}
	for every $x\in  S \cap\bar B(\bar x,\delta_0)$. 
	
	Now, let $\{\bar u_{k+1},\dots,\bar u_m\}$ be a basis of ${\rm Ker}h'(\bar x)^*$. By \eqref{equality_of_kernel},  $\{\bar u_{k+1},\dots,\bar u_m\}$ is also  a basis of ${\rm Ker}h'( x)^*$ for every $x\in  S$. Add vectors $\bar u_1,\dots,\bar u_k$ to get a basis $\{\bar u_1,\dots,\bar u_m\}$  of $\mathbb{R}^m$.  
	Set
	\begin{equation}\label{z_representation}
	\bar z_i = h'(\bar x)^*\bar u_i
	\end{equation}
	and
	\begin{equation}\label{w_representation}
	w_i(x) = h'( x)^*\bar u_i
	\end{equation}
	for $i=1,\dots,k$.
	Then, we have that $\{\bar z_1,\dots,\bar z_k\}$ and $\{w_1(x),\dots,w_k(x)\}$ are bases of ${\rm Im}h'(\bar x)^*={\rm Im}h'(x)^*$ for every $x\in  S \cap\bar B(\bar x,\delta_0)$.  
	
	We extend $\{\bar z_1,\dots,\bar z_k\}$ to obtain a basis of $\{\bar z_1,\dots,\bar z_n\}$ of $\mathbb{R}^n$. By Lemma \ref{Boundedness_of_norm_of_coefficients}, there exists $\varepsilon>0$ such that $\{z_1,\ldots,z_n\}$ is also a basis of $\mathbb{R}^n$ for every $z_i\in \bar B(\bar z_i,\varepsilon)$ with $i=1,\ldots,n$. Moreover, since $h'$ is continuous at $\bar x$, we have $\displaystyle\lim_{x\to \bar x}\|h'(x)^*-h'(\bar x)^*\|=0$. Hence, there is a number $\delta\in (0,\delta_0)$ such that $\|h'(x)^*-h'(\bar x)^*\|<\frac{\varepsilon}{d}$, where 
	$$d:=\max\{\|\bar u_1\|,\dots,\|\bar u_k\|\}>0,$$  
	for every $x\in  S \cap\bar B(\bar x,\delta)$. Combining this with \eqref{z_representation} and \eqref{w_representation}, we get
	\begin{equation}\label{z_w_norm_relation}
	\|w_i(x)-\bar z_i\| = \|(h'(x)^*-h'(\bar x)^*)\bar u_i\|\le \|(h'(x)^*-h'(\bar x)^*)\|\|\bar u_i\|\le (\frac{\varepsilon}{d})d=\varepsilon,
	\end{equation}
	for every $x\in  S \cap\bar B(\bar x,\delta),\, i=1,\ldots,k$.
	
	Now, fix any $x\in  S \cap \bar B(\bar x,\delta)$. Recall that $\{z_1,\ldots,z_n\}$ is a basis of $\mathbb{R}^n$ for every $z_i\in \bar B(\bar z_i,\varepsilon),\, i=1,\ldots,n$. Hence, invoking \eqref{z_w_norm_relation}, we can extend the system  
	$\{w_{1}(x),\dots,w_k(x)\}$
	to get a basis 
	$\{w_{1}(x),\dots,w_n(x)\}$
	of $\mathbb{R}^n$ so that $\|w_i(x)-\bar z_i\|\le \varepsilon$ for every $i=1,\dots,n$. Thus, there is a (unique)  linear bijection $P(x):\, \mathbb{R}^n\to\mathbb{R}^n$ such that
	$$  \bar z_i = P(x)(w_i(x)), \, i=1,\dots n.  $$
	Furthermore, we have
	$$(P(x)\,o\, h'( x)^*)\bar u_i = P(x)(w_i(x)) =\bar  z_i = h'(\bar x)^*\bar u_i, i=1,\ldots, k$$
	and
	$$
	(P(x)\,o \,h'( x)^*)\bar u_j = 0 =  h'(\bar x)^*\bar u_j, j=k+1,\ldots, n.
	$$
	Therefore, 
	\begin{equation}\label{derivative_of_h_isomorphism}
	h'(\bar x)^*=P(x)\, o\, h'( x)^*
	\end{equation} 
	for every $x\in  S \cap\bar B(\bar x,\delta)$.
	
	Next, we will show that there exists a positive number $M$ such that 
	\begin{equation}\label{P_x_norm_upper_bound}
	\|P(x)\|\le M
	\end{equation}
	for every $x\in  S \cap \bar B(\bar x,\delta)$. 
	Since $\{w_1(x),\cdots,w_n(x)\}$ is a basis of $\mathbb{R}^n$, for every $w\in \mathbb{R}^n$, we have the representation 
	$$w=\sum_{i=1}^n \alpha_i(x) w_i(x).$$ 
	Invoking  Lemma \ref{Boundedness_of_norm_of_coefficients} again, we get the set
	$$ \bigg\{(\alpha_1(x),\ldots,\alpha_n(x))^T\in\mathbb{R}^n: \|w\|=\|\sum_{i=1}^n \alpha_i(x) w_i(x)\|\le1,\, x\in  S \cap \bar B(\bar x,\delta)\bigg\}$$
	is bounded by a positive number $m_0$, which does not depend on $x\in  S \cap \bar B(\bar x,\delta)$.
	Hence, fixing  any $x\in  S \cap \bar B(\bar x,\delta)$, for every $w\in\mathbb{R}^n$ satisfying $\|w\|\le1$, we have
	\begin{align*}
	\|(P(x)-I)w\|&=\big\|\sum_{i=1}^n\alpha_i(x)(P(x)-I)w_i(x)\big\|=\big\|\sum_{i=1}^n\alpha_i(x     )(P(x)w_i(x)-w_i(x))\big\|\\
	&=\big\|\sum_{i=1}^n\alpha_i(x)(z_i-w_i(x))\big\|\le\sum_{i=1}^n |\alpha_i(x)|\|z_i-w_i(x)\|\\
	&\le \varepsilon\sum_{i=1}^n |\alpha_i(x)|\\
	&\le \varepsilon\sqrt{n}\big\|(\alpha_1(x),\cdots,\alpha_n(x))^T\big\|\\
	&\le \varepsilon\sqrt{n}m_0,
	\end{align*}
	where $I$ denotes the identity operator in $\mathbb{R}^n$.
	It follows that
	$$\|P(x)-I\|=\sup_{\|w\|\le1}\|(P(x)-I)w\| \le \varepsilon\sqrt{n}m_0.$$
	Thus,
	$$\|P(x)\|\le \|I\|+ \varepsilon\sqrt{n}m_0=1+ \varepsilon\sqrt{n}m_0$$
	for all $x\in  S \cap \bar B(\bar x,\delta)$.
	Putting $M= 1+ \varepsilon\sqrt{n}m_0,$
	we obtain \eqref{P_x_norm_upper_bound}.
	
	Now, note that ${\rm Ker}\,h'(x)^*\cap {\rm int}K^\circ\ne \emptyset$ and that $h'(x)^*(K^\circ)=N_S(x)$ is closed for every $x\in  S \cap \bar B(\bar x,\delta)$. Hence, by Lemma \ref{lsc_of_multifunction}, for every $x\in  S \cap \bar B(\bar x,\delta)$, there exists $\tau_x>0$ such that
	\begin{equation*}\label{inclusion_for_x}
	L^\perp\cap \bar B_{\mathbb{R}^n}=h'(x)^*(K^\circ)\cap \bar B_{\mathbb{R}^n}\subset \tau_x h'(x)^*(K^\circ\cap \bar B_{\mathbb{R}^m}),
	\end{equation*}
	where the equation is due to \eqref{Image_sigma_normal_cone_equation}. This implies that the zero vector is an interior point of $h'(x)^*(K^\circ\cap\bar B_{\mathbb{R}^m})$ restricted to $L^\perp$.
	In addition, the set $h'(x)^*(K^\circ\cap\bar B_{\mathbb{R}^m})$ is  convex and compact.  Thus, 
	\begin{equation*}\label{inclusion_for_h_x}
	m_x:=\min \big\{ \|v\|: u\in  {\rm bd} _{L^\perp}(h'(x)^*(K^\circ\cap\bar B_{\mathbb{R}^m}))\big\}>0
	\end{equation*}
	for every $x\in  S \cap \bar B(\bar x,\delta)$. (Here, ${\rm bd}_{L^\perp} (D)$ denotes the boundary of $D$ in the subspace $L^\perp$.) It follows that
	\begin{equation}\label{inclusion_for_h_x_2}
	h'(x)^*(K^\circ)\cap \bar B_{\mathbb{R}^n}= L^\perp\cap \bar B_{\mathbb{R}^n}\subset \max\{\frac{1}{m_x},1\} h'(x)^*(K^\circ\cap \bar B_{\mathbb{R}^m})
	\end{equation}
	for every $x\in  S \cap \bar B(\bar x,\delta)$.
	
	Fixing $x\in  S \cap \bar B(\bar x,\delta)$, there exists a vector $v_0(x)\in  {\rm bd} _{L^\perp}(h'(x)^*(K^\circ\cap\bar B_{\mathbb{R}^m})$ such that $\|v_0(x)\|=m_x$.  Noting that  $h'(x)^*(K^\circ\cap\bar B_{\mathbb{R}^m})\subset L^\perp$, by \eqref{derivative_of_h_isomorphism},  we have 
	$$h' ( \bar x)^*(K^\circ\cap\bar B_{\mathbb{R}^m})=P(x)\big[h' ( x)^*(K^\circ\cap\bar B_{\mathbb{R}^m})\big]=P(x)|_{L^\perp}\big[h' ( x)^*(K^\circ\cap\bar B_{\mathbb{R}^m})\big].$$ 
	Noting also that  $P(x)|_{L^\perp}$ is  a linear bijection from $L^\perp$ onto ${\rm Im}(P(x)|_{L^\perp})$, by  Lemma \ref{Boundary_point_of_set_isomorphism}, we can find a vector $\bar v_0\in  {\rm bd} _{L^\perp} \big(h'(\bar x)^*(K^\circ\cap\bar B_{\mathbb{R}^m})\big)$ satisfying 
	$$\bar v_0=P(x)|_{L^\perp}( v_0(x))=P(x)( v_0(x)).$$ 
	It follows that 
	$$\|\bar v_0\|\le\|P(x)\|\| v_0(x)\|.$$ 
	Thus, by \eqref{P_x_norm_upper_bound}, 
	$$m_x=\| v_0(x)\|\ge\|\bar v_0\|/\|P(x)\|\ge m_{\bar x}/M$$ 
	for every $x\in  S \cap \bar B(\bar x,\delta)$.  Combining this with \eqref{inclusion_for_h_x_2}, we obtain
	
	$$h'(x)^*(K^\circ)\cap\bar B_{\mathbb{R}^n}\subset \tau h'(x)^*(K^\circ\cap\bar B_{\mathbb{R}^m}),$$
	where $\tau =\max\{\frac{M}{m_{\bar x}},1\}$, for every $x\in  S \cap \bar B(\bar x,\delta)$. Hence, we deduce that \eqref{general_inclusion} is satisfied for every $x\in  S \cap \bar B(\bar x,\delta)$. Moreover, since $\delta<\delta_0$, \eqref{abadie_general_case} is also satisfied for every $x\in  S \cap \bar B(\bar x,\delta)$. So, in Subcase 2.1, the inclusion \eqref{convex_cone_inclusion_pre} has a local error bound at $\bar x$.
	
	\medskip
	\noindent {\bf Subcase 2.2.} $\{0\}\ne{\rm Ker}\,h'(\bar x)^*\cap K^\circ\subset  {\rm bd}  K^\circ$.
	
	Note that  $K^\circ$ is strictly convex due to Proposition \ref{smooth_cone_then_strictly_convex_1}. Combining this with the assumption in this subcase, by Lemma \ref{image_of_pointed_cone}, we derive $h'(\bar x)^*(K^\circ)$ is a pointed cone. Since the ${\rm ACQ}( S \cap U$) holds for \eqref{convex_cone_inclusion_pre}, it is implied that $N_ S (\bar x)=h'(\bar x)^*(K^\circ)$. Hence, $N_ S (\bar x)$ is a pointed cone. Thus, from the claim in \cite[p. 7]{Auslender_Teboulle_2003} we infer that ${\rm int} S \ne\emptyset$. Moreover, it is clear that ${\rm Ker}\,h'(\bar x)^*\ne \mathbb{R}^m$, yielding $h'(\bar x)\ne 0$. Since $h'$ is continuous at $\bar x$, there exists $\delta_0>0$ such that $h'(x)\ne 0$ for every $x\in \bar B(\bar x,\delta_0)$.
	
	Next, from the assumption, by invoking Proposition \ref{interrior_perpendicular_cone}(ii),  we can deduce  that $\{0\}\ne{\rm Im}h'(\bar x)\cap K\subset  {\rm bd}  K$. Combining this with the fact that $K$ is strictly convex, we infer that ${\rm Im}h'(\bar x)\cap K$ is a ray of $ {\rm bd}  K$, denoted by $\ell^+$. Additionally, using  \eqref{image_equation} with $u:=\bar x$, we see that for every $x\in S $, there exists $w(x)\in K$ satisfying
	\begin{equation}\label{image_cone_ray}
	h(x)+w(x)=h'(\bar x)(x-\bar x)\in {\rm Im}h'(\bar x).
	\end{equation}
	From this and the fact that $h(x)+w(x)\in K$ we get
	$h(x)+w(x)\in {\rm Im}h'(\bar x)\cap K=\ell^+,$ 
	yielding $(h(x)+w(x))/2\in \ell^+$.
	If $h(x)$ and $w(x)$ are not on the same ray, then by \eqref{strictly_Convex_interior}, we derive $(h(x)+w(x))/2\in {\rm int }K$. Hence, $\ell^+\cap {\rm int}K\ne\emptyset$, a contradiction. Thus, $h(x)$ and $w(x)$ are on the same ray. Combining this with \eqref{image_cone_ray}, we derive that $h(x)\in \ell^+$ for every $x\in S$. It follows that
	\begin{equation}\label{S_representation}
	S =\{x\in\mathbb{R}^n:\, h(x)\in  \ell^+ \}.
	\end{equation}
	
	Now, denote by $\{e_1,\ldots,e_m\}$ an orthogonal basis of $\mathbb{R}^m$ such that $e_1\in\ell^+$. In this basis, we have the representation $h(x)=(h_1(x),\cdots,h_m(x))^T$ for every $x\in \mathbb{R}^n$. Then, using \eqref{S_representation}, we obtain
	\begin{equation*}\label{Sigma_repenstation}
	S =\{x\in\mathbb{R}^n:\, h_1(x)\ge0,h_2(x)=\cdots=h_m(x)=0\}.
	\end{equation*}
	We will show that, for each $k=2,\ldots,m$, 
	\begin{equation}\label{Derivative_of_h_k}
	h_k'(x)=0\, \forall x\in S.
	\end{equation}
	Fixing $k\in\{2,\ldots,m\}$ and $\tilde x\in S $, since $S$ is convex and ${\rm int}S\ne\emptyset$, there exists a number $\delta_1\in (0,\delta_0)$ such that $B(\tilde x,\delta_1)\cap {\rm int} S \ne \emptyset$. Hence, we can choose a point $x^0\in B(\tilde x,\delta_1)\cap {\rm int} S $ and a number $\delta_2\in(0,\delta_1)$ such that $B(x^0,\delta_2)\subset B(\tilde x,\delta_1)\cap {\rm int} S $. Thus, for every $x\in B(x^0,\delta_2)$ and all non-zero real numbers $\lambda$ sufficiently small, we derive
	$$h_k(\tilde x +\lambda (x-\tilde x))=h_k(\tilde x)+\lambda h_k'(\tilde x)(x-\tilde x)+o(\lambda),$$
	with $\displaystyle \lim_{\lambda\to 0}\frac{o(\lambda)}{\lambda}=0$.
	From this and the fact that $h_k(\tilde x +\lambda (x-\tilde x))=h_k(\tilde x)=0$ we get
	$h_k'(\tilde x)(x-\tilde x)=0$
	for every $x\in B(x^0,\delta_2)$.
	It follows that $h_k'(\tilde x)[B(x^0-\tilde x,\delta_1)]=\{0\}$. This yields $B(x^0-\tilde x,\delta_1)\subset {\rm Ker}h_k'(\tilde x)$. Therefore, ${\rm Ker}h_k'(\tilde x)=\mathbb{R}^n$, and so \eqref{Derivative_of_h_k} holds.
	
	Next, by representing $h'(x)=(h'_{ij}(x))_{m\times n}$ for every $x\in \mathbb{R}^n$, it follows from \eqref{Derivative_of_h_k} that $h'_{ij}(x)=0$ for every $x\in  S $, $i=2,\ldots,m,j=1,\ldots n$. Hence,
	\begin{equation}\label{image_of_normal_cone_equation_0}
	h'(x)^*[N_K(h(x))]=\big\{ (h_{11}'(x)z_1,\cdots,h_{1n}'(x)z_1)^T:\,z=(z_1,\cdots,z_m)^T\in N_K(h(x)) \big\}
	\end{equation}
	for every $x\in  S $.
	This infers that 
	\begin{equation}\label{image_normal_cone_equation_2}
	\begin{aligned}
	&h'(x)^*[N_K(h(x))\cap\bar B_{\mathbb{R}^m}]\\
	&=\Bigg\{ (h_{11}'(x)z_1,\cdots,h_{1n}'(x)z_1)^T:z\in N_K(h(x))\cap\bar B_{\mathbb{R}^m}\Bigg\}
	\end{aligned}
	\end{equation}
	for every $x\in  S $.
	
	It is not difficult to see that for every $x\in S$ and $z=(z_1,\ldots,z_m)^T\in N_K(h(x))$, we have $z_1\le0$.
	Hence, \eqref{image_of_normal_cone_equation_0} implies that
	\begin{equation}\label{image_normal_cone_equation}
	\begin{aligned}
	&h'(x)^T[N_K(h(x))]\cap\bar B_{\mathbb{R}^n}\\
	&=\Bigg\{ (h_{11}'(x)z_1,\cdots,h_{1n}'(x)z_1)^T: z\in N_K(h(x)),\,(-z_1)\sqrt{\sum_{j=1}^n h'_{1j}(x)^2}\le 1 \Bigg\}.
	\end{aligned}
	\end{equation}
	for every $x\in  S $.
	
	Recall that $h'(x)\ne0$ for every $x\in \bar B(\bar x,\delta_0)$ and that $h'_{ij}(x)=0$ for every $x\in  S $, $i=2,\ldots,m,\,j=1,\ldots n$. Consequently, observing that $\delta_2<\delta_1<\delta_0$, we get $\displaystyle\sqrt{\sum_{j=1}^n h'_{1j}(x)^2}>0$ for every $x\in  S \cap \bar B(\bar x,\delta_2)$. Moreover, since $h$ is continuously differentiable and $S \cap \bar B(\bar x,\delta_2)$ is compact, it is inferred that 
	
	$$m_1:=\min\Bigg\{\sqrt{\sum_{j=1}^n h'_{1j}(x)^2}:\, x\in  S \cap \bar B(\bar x,\delta_2)\Bigg\}>0.$$
	Combining this with \eqref{image_normal_cone_equation}, we obtain
	\begin{equation}\label{image_normal_cone_inclusion}
	h'(x)^*[N_K(h(x))]\cap\bar B_{\mathbb{R}^n}\subset\big   \{ (h_{11}'(x)z_1,\ldots,h_{1n}'(x)z_1)^T:
	z_1\in [-1/m_1,0]\big\}
	\end{equation}
	for every $x\in S\cap\bar B(\bar x,\delta_2)$.
	
	By \eqref{S_representation}, for every $x\in  S $, if $h(x)=0$ then $N_K(h(x))=K^\circ$; if $h(x)\ne 0$ then $N_K(h(x))=\ell^-$, a constant ray in $K^\circ$, due to  Proposition \ref{regular_cone}(ii). Furthermore, from \eqref{normal_cone_form} it is not difficult to see that $\langle u,v\rangle=0$ for every $u\in \ell^-$ and $v\in \ell^+$. This means that
	$$C_1:=\{z_1:\, z=(z_1,\ldots,z_m)^T\in \ell^-\}=\{0\}.$$
	The latter, in conjunction with \eqref{image_of_normal_cone_equation_0}, implies that  
	\begin{equation}\label{image_normal_cone_zero}
	h'(x)^*[N_K(h(x))]=\{0\}
	\end{equation}
	for all $x\in S$ satisfying $h(x)\ne0$.
	
	On the other hand, by putting
	$$C_2=\{z_1:\, z=(z_1,\ldots,z_m)^T\in K^\circ\cap\bar B_{\mathbb{R}^m}\},$$
	we have $C_2\ne\{0\}$.  Additionally, $C_2$ is connected, compact and $z_1\le0$ for every $z_1\in C_2$. It follows that $ m_2:=\min C_2<0$ and 
	$C_2=\{z_1:\, z_1\in[m_2,0]\}.$
	Combining this with \eqref{image_normal_cone_equation_2}, we get 
	\begin{equation*}\label{image_normal_cone_equation_22}
	h'(x)^*[N_K(h(x))\cap\bar B_{\mathbb{R}^m}]
	=\Bigg\{ (h_{11}'(x)z_1,\ldots,h_{1n}'(x)z_1)^T:z_1\in[m_2,0]\Bigg\}
	\end{equation*}
	for all $x\in S$ satisfying $h(x)=0$.
	Based on this,  and setting $\tau=-\frac{1}{m_1m_2}>0$, along with \eqref{image_normal_cone_inclusion}, we deduce that \eqref{general_inclusion} holds for all $x\in S\cap \bar B(\bar x,\delta_2)$ satisfying $h(x)=0$. In addition, by \eqref{image_normal_cone_zero}, it is evident that \eqref{general_inclusion} also holds for all $x\in S\cap \bar B(\bar x,\delta_2)$ satisfying $h(x)\ne0$. In summary, with $\tau=-\frac{1}{m_1m_2}>0$,  we conclude that \eqref{general_inclusion} is satisfied for all $x\in S\cap \bar B(\bar x,\delta_2)$.
	
	Moreover, since the ${\rm ACQ}( S \cap U$) holds for \eqref{convex_cone_inclusion_pre}, there exists $\delta\in(0,\delta_2)$ such that \eqref{abadie_general_case} holds for every $x\in  S \cap \bar B(\bar x,\delta)$. Clearly,  \eqref{general_inclusion} also holds for every $x\in S\cap \bar B(\bar x,\delta)$. So, we have proved that, the inclusion \eqref{convex_cone_inclusion_pre} has a local error bound at $\bar x$.
	
	The proof is complete. $\hfill\Box$
\end{proof}

\medskip
The following example shows that, in general, the ACQ($S$) does not imply the existence of a global error bound for \eqref{convex_cone_inclusion_pre}. 
\begin{example}(cf. \cite[Example 4.1]{Li_2010})\label{Example_1}
	Consider the continuously differentiable function $h=(h_1,h_2): \mathbb{R}^4\to \mathbb{R}^2$ with its components defined as:
	$$h_1(x_1,x_2,x_3,x_4)= -x_1$$ and
	\begin{align*}
	h_2(x_1,x_2,x_3,x_4)=-x_1^{16}-x_2^8-x_3^6-x_1x_2^3x_3^3-x_1^2x_2^4x_3^2&-x_2^2x_3^4-x_1^4x_3^4-x_1^4x_2^6
	\\
	&-x_1^2-x_2^2-x_3^2+x_4.	
	\end{align*}
	Setting $K=\mathbb{R}_+^2 $, we have $K$ is a smooth cone in $\mathbb{R}^2$.

	According to \cite[Example 4.1]{Li_2010}, the functions $h_1$ and $h_2$ are concave.  Hence, the function $h$ is $K$-concave. Thus, the inclusion $h(x)\in K$ satisfies all the conditions in Theorem \ref{nonlinear_case_theorem}.
	
	Note that $h(1,0,0,3)=(1,1)\in {\rm int}K$. Therefore,  the inclusion satisfies the SCQ. Consequently, the ACQ($ S $) is satisfied.
	
	On the other hand, putting $f=\max\{-f_1,-f_2\}$, we get 
	$$ S =\{x\in \mathbb{R}^4: h(x)\in K\}=\{x\in \mathbb{R}^4: f(x)\le 0\}. $$ 
	We also have $d(h(x)| K)= \max(f(x),0)$ for all $x\in\mathbb{R}^4$.  Therefore, the existence of a global error bound for the inclusion is equivalent to the existence of  a Lipschitz-type global error bound of $f$ defined in \cite[Definition 4.1]{Li_2010}. However, by \cite[Example 4.1]{Li_2010}, the function $f$ has no Lipschitz-type global error bounds. It follows that the inclusion has no global error bounds.

\end{example}

\section{Global Error Bounds for Smooth Cone-Affine Inclusions}
In this  section, we consider the case where $h(x)=Ax+b$, with $A: \mathbb{R}^{n}\to  \mathbb{R}^m$ being a linear function and $b\in \mathbb{R}^m$, and assume that the cone $K$ is smooth.  Then, the inclusion \eqref{convex_cone_inclusion_pre} and its solution set become
\begin{equation}\label{convex_cone_inclusion_linear_case}
Ax+b\in K.
\end{equation}
and
\begin{equation}\label{sigma_define_linear_case}
S =\{x \in  Y\, : Ax+b\in K \big\},
\end{equation}
respectively.
Moreover, the ACQ($ S$) for this inclusion simplifies to
\begin{equation}\label{abadie_linear_case} 
A^*[N_{K}(Ax+b)]=N_{ S }(x)\, \forall x\in  S.
\end{equation}
By invoking Theorem \ref{error_bound_equivalence}, we derive that the inclusion \eqref{convex_cone_inclusion_linear_case} has a global error bound if and only if  the ACQ($S$) holds for \eqref{convex_cone_inclusion_linear_case} and there exists a constant $\tau>0$ such that
\begin{equation}\label{linear_inclusion}
A^*\big[N_{K}(Ax+b)\big]\cap \bar B_{\mathbb{R}^n}\subset \tau A^*\big[N_{K}(Ax+b)\cap \bar B_{\mathbb{R}^m}\big]\,\forall x\in S.
\end{equation}

As mentioned earlier, the authors in \cite{Ng_Yang_2002} provide conditions for the existence of global error bounds for \eqref{convex_cone_inclusion_linear_case} when 
$K$ is a second-order cone. We will extend these results to the scenario where 
$K$  is a smooth cone. In doing so, we precisely identify the cases in which \eqref{convex_cone_inclusion_linear_case} admits a global error bound. Our proofs rely extensively on the relationship between error bounds and the Abadie constraint qualification. Consequently, we also determine the exact conditions under which the ACQ($S$) holds for \eqref{convex_cone_inclusion_linear_case}.

\begin{theorem}\label{affine_case_theorem_1} If ${\rm Im}A\cap {\rm int}K\ne\emptyset$, then the inclusion \eqref{convex_cone_inclusion_linear_case} has a global error bound.
\end{theorem}
\begin{proof}
	Note that $({\rm Im}A)^\perp={\rm Ker}A^*$ and $(K^\circ)^\circ=K$. Hence, from the assumption and by invoking Proposition \ref{interrior_perpendicular_cone}, we obtain ${\rm Ker}A^*\cap K^\circ=\{0\}$. Applying \cite[Theorem 5]{Burke_Deng_2009}, it follows that \eqref{convex_cone_inclusion_linear_case} has a global error bound. $\hfill\Box$
\end{proof}	

\begin{theorem}\label{affine_case_theorem_2} Assume that ${\rm Im}A\cap K=\{0\}$. Then the inclusion \eqref{convex_cone_inclusion_linear_case} has a global error bound  if and only if   ACQ($ S $) holds for \eqref{convex_cone_inclusion_linear_case}. 	
	Moreover, the following statements hold:
	
	(i)	 If $ {(\rm Im}A+b)\cap {\rm int}K\ne \emptyset$, then \eqref{convex_cone_inclusion_linear_case} has a global error bound;

	(ii) If  $b\in {\rm Im}A$, then \eqref{convex_cone_inclusion_linear_case} has a global error bound;
	
	(iii) If $ {(\rm Im}A+b)\cap {\rm int}K= \emptyset$ and $b\notin {\rm Im}A$, then \eqref{convex_cone_inclusion_linear_case} has no global error bounds.
	
\end{theorem}
\begin{proof}
	To show the equivalence between the existence of global error bounds for \eqref{convex_cone_inclusion_linear_case} and the satisfaction of ACQ($ S $) for the inclusion,  we need only to show that if ACQ($ S $) holds for \eqref{convex_cone_inclusion_linear_case}, then the inclusion admits a global error bound.
	
	For every $x\in S$, since ACQ($ S $) holds, we have the ACQ($S\cap U$) holds for some neighborhood of $x$. By Theorem \ref{nonlinear_case_theorem}, \eqref{convex_cone_inclusion_linear_case} has a local error bound at $x$. This implies that  \eqref{convex_cone_inclusion_linear_case} has a local error bound at every point of $ S $. Thus, by applying \cite[Proposition 3]{Burke_Deng_2009}, we arrive at the conclusion.
	
	Using the above equivalence, to prove the statements (i)-(iii), we need to investigate the existence of ACQ($ S $) for \eqref{convex_cone_inclusion_linear_case}.

	(i) Since $ {(\rm Im}A+b)\cap {\rm int}K\ne \emptyset$, the SCQ holds for \eqref{convex_cone_inclusion_linear_case}, which yields the ACQ($ S $) holds for \eqref{convex_cone_inclusion_linear_case}.

	(ii) From the assumption and the fact that $b\in {\rm Im}A$, we derive $$S=\{x\in \mathbb{R}^n:\, Ax+b\in K\}=\{x\in \mathbb{R}^n:\, Ax+b=0\}.$$
	Then, putting $b=-A\bar x$ with $\bar x\in \mathbb{R}^n$, we get $S=\bar x+{\rm Ker}A$. Hence, for any $x\in S$, we have $N_S(x)={\rm Im}A^*$. Moreover, it is clear that $A^*[N_K(Ax+b)]=A^*(K^\circ)$ for all $x\in S$. 
	
	Clearly, $A^*(K^\circ)\subset {\rm Im}A^*$. We also have ${\rm Im}A^*\subset A^*(K^\circ)$. Indeed, taking any $y\in{\rm Im}A^*$, there exists $ z\in \mathbb{R}^n$ satisfying $y=A^* z$.  By the assumption and Proposition \ref{interrior_perpendicular_cone}, we deduce that ${\rm Ker}A^*\cap {\rm int}K^\circ\ne \emptyset$. Thus, we can find a vector $\bar u\in {\rm Ker}A^*\cap {\rm int}K^\circ$. This yields $A^*\bar u=0$ and $\bar u+z/\lambda\in {\rm int} K^\circ$ for some sufficiently large positive number $\lambda$. Combining this with the fact that $K^\circ$ is a cone, we derive $\lambda(\bar u+z/\lambda)=\lambda \bar u+z\in  K^\circ$. Moreover,
	$$A^*(\lambda \bar u+z)=\lambda A^*(\bar u)+A^*(z)=A^*(z)=y.$$ 
	It follows that $y\in A^*(K^\circ)$, implying ${\rm Im}A^*\subset A^*(K^\circ)$. Therefore, $A^*(K^\circ)= {\rm Im}A^*$. This infers that \eqref{abadie_linear_case} holds. So, ACQ($S$) holds for \eqref{convex_cone_inclusion_linear_case}.
	
	(iii) Since $S\ne\emptyset$, we have ${(\rm Im}A+b)\cap K\ne\emptyset$. Using similar arguments as in the proof of \cite[Theorem 5.8 (c)]{Ng_Yang_2002}, we will show that ${(\rm Im}A+b)\cap K$ contains only one non-zero vector, denoted by $\bar z$. For any  $z\in{(\rm Im}A+b)\cap K$, we have $0\notin [\bar z,z]$, as $b\notin {\rm Im}A$, and $ [\bar z,z]\subset {\rm bd}K$, because  $ {(\rm Im}A+b)\cap {\rm int}K= \emptyset$. Combining this with Proposition \ref{strictly_convex_cone_0}(ii) yields $\bar z$ and $z$ lie on the same ray of ${\rm bd}K$. This means that there exists $\lambda>0$ satisfying $z=\lambda \bar z$. If $\lambda>1$, then $0\ne z-\bar z\in {\rm Im}A\cap K$, which is a contradiction. Similarly, if $\lambda<11$, then $0\ne \bar z- z\in {\rm Im}A\cap K$, which is also a contradiction. Hence, $\lambda=1$, implying that $z=\bar z$. Thus, ${(\rm Im}A+b)\cap K=\{\bar z\}$.
	
	It follows from the latter assertion that 
	$$S=\{x\in \mathbb{R}^n:\, Ax+b\in K\}=\{x\in \mathbb{R}^n:\, Ax+b=\bar z\}=\bar u+{\rm Ker }A,$$
	where $\bar u\in \mathbb{R}^n$ satisfying $A\bar u=\bar z-b$, and that $N_K(Ax+b)=N_K(\bar z)$ is a ray for all $x\in S$.
	Therefore, $N_S(x)={\rm Im}A^*$ and $A^*[N_K(Ax+b)]$ is at most a ray for all $x\in S$. Consequently, the ACQ($S$) does not hold for \eqref{convex_cone_inclusion_linear_case}. $\hfill\Box$
\end{proof}

\begin{theorem}\label{affine_case_theorem}  Assuming that $\{0\}\ne{\rm Im}A\cap K\subset  {\rm bd}  K$, we have the following:
	
	(i)	 If ${\rm dim} {(\rm Im}A)\ne1$  and $b\in {\rm Im}A$, then ACQ($S$) does not hold for \eqref{convex_cone_inclusion_linear_case}. Consequently, \eqref{convex_cone_inclusion_linear_case} has no global error bounds;

	(ii) If ${\rm dim} {(\rm Im}A)\ne1$  and $b\notin {\rm Im}A$, then ACQ($S$)  holds for \eqref{convex_cone_inclusion_linear_case}, but  \eqref{convex_cone_inclusion_linear_case} has no global error bounds;
	
	(iii) If ${\rm dim} {(\rm Im}A)=1$, then \eqref{convex_cone_inclusion_linear_case} has a global error bound.
\end{theorem}
\begin{proof}
	
	Since $K$ is strictly convex, it follows from the assumption and Proposition \ref{strictly_convex_cone}(i) that  ${\rm Im}A\cap K$ is a ray  in $ {\rm bd}  K$, denoted by $\ell^+$.    
	
	(i)  As $b\in {\rm Im}A$, it follows that $Ax+b\in {\rm Im}A$ for every $x\in\mathbb{R}^n$. Hence,   \begin{equation}\label{S_Affine_0}
	S=\{x\in\mathbb{R}^n:\,Ax+b\in K\}=\{x\in\mathbb{R}^n:\,Ax+b\in \ell^+\}.
	\end{equation}
	
	We have ${\rm int}S=\emptyset$. Indeed, if the claim is false, then there exist $\bar x\in S$ and $\varepsilon>0$ such that $B(\bar x,\varepsilon)\subset S$. Thus, for all $x\in B(\bar x,\varepsilon)$, it follows from \eqref{S_Affine_0} that
	\begin{equation}\label{S_Affine_1}
	A(x-\bar x)=(Ax+b)-(A\bar x+b)\in \ell^+-\ell^+.
	\end{equation}
	In addition, taking any $x\in \mathbb{R}^n\backslash\{\bar x\}$, we have $\varepsilon(x-\bar x)/\|x-\bar x\|\in B(\bar x,\varepsilon)$. By \eqref{S_Affine_1} and the fact that $\ell:=\ell^+-\ell^+$ is a line through the zero vector, we obtain $A(x-\bar x)\in \ell$, inferring $Ax\in A\bar x+\ell$. 
	This implies that ${\rm dim} {(\rm Im}A)=1$, which is a contradiction. Hence, ${\rm int}S=\emptyset$.  Therefore, $N_S(x)$ is a linear subspace of $\mathbb{R}^n$ with dimension at least 1 for every $x\in S$.
	
	On the other hand, for some $x\in S$ with $Ax+b\in \ell^+\backslash\{0\}$, since $N_K(Ax+b)$ is a ray, we have that $A^*[N_K(Ax+b)]$ is at most a ray. It follows that $N_S(x)\ne A^*[N_K(Ax+b)]$. This means that the ACQ($S$) does not holds for \eqref{convex_cone_inclusion_linear_case}. As a result, the inclusion has no global error bounds.
	
	(ii) 	First, we claim that 	
	\begin{equation}\label{SCQ_Affine}
	(b+{\rm Im}A)\cap {\rm int}K\ne\emptyset.
	\end{equation} Indeed, if the claim is false, then $(b+{\rm Im}A)\cap K\subset {\rm bd}K$. Since $S\ne\emptyset$ and $b\notin {\rm Im}A$, there exists a non-zero vector $\bar y$ such that $\bar y\in (b+{\rm Im}A)\cap K\subset {\rm bd}K$. If $\bar y\notin \ell^+$, then taking an arbitrary non-zero vector $y\in \ell^+$, we get $\bar y+y=2\frac{\bar y+y}{2}\in {\rm int}K$, due to \eqref{strictly_Convex_interior} and the fact that $K$ is a cone, a contradiction. Accordingly, $\bar y\in \ell^+\subset {\rm Im}A$, implying that $b+{\rm Im}A=\bar y+{\rm Im}A={\rm Im}A$. This yields $b\in {\rm Im}A$, contradicting our assumption. 
	
	From \eqref{SCQ_Affine} we infer that the SCQ holds for \eqref{convex_cone_inclusion_linear_case}. Hence, the  ACQ($S$) holds for \eqref{convex_cone_inclusion_linear_case}.
	
	Next, it is not difficult to see that there exists a vector $z^1\in (b+{\rm Im}A)\backslash K$, since $K$ is pointed. Then, taking an arbitrary $z^2\in(b+{\rm Im}A)\cap {\rm int}K$, the segment connecting $z^1$ and $z^2$ will intersect ${\rm bd}K$ at a point denoted by $\bar z$. Clearly, $\bar z\in b+{\rm Im}A$,  $\bar z\notin {\rm Im}A$ and $\bar z+\ell^+\subset K$. Combining the latter with the fact that $\bar z+\ell^+\subset b+{\rm Im}A$ yields $\bar z+\ell^+\subset (b+{\rm Im}A)\cap K$. Moreover, setting $M={\rm cone}(\bar z +\ell^+)$, we can deduce that $\ell^+\subset{\rm cl} M$.  This implies that $\ell^+\subset {\rm cl}{N}$, where $N:={\rm cone}((b+{\rm Im}A)\cap K)\supset M$.

	We will show that ${\rm cl} N=N\cup \ell^+$. Clearly, ${\rm cl} N\supset N\cup \ell^+$.  Now, consider a sequence $(t_ky^k)\subset N$ converging to $\bar z\in {\rm cl} N\backslash N$, where $t_k\ge 0$ and $y_k\in (b+{\rm Im}A)\cap K$ for every $k=1,2,\ldots$.	
	Then,  $\bar z\in V\cap K$, where $V$ is the linear subspace of $\mathbb{R}^m$ generated by $\{b\}\cup{\rm Im}A$. Since $\bar z\in K\backslash N$, we obtain $\bar z\notin {\rm cone}({\rm Im}A+b)$.
	This yields 
	\begin{equation}\label{ray_subspace_intersection}\{t\bar z:t\ge0\}\cap  ({\rm Im}A+b)=\emptyset.
	\end{equation}
	Moreover, since $b\notin {\rm Im}A$, we see that    ${\rm Im}A$ is a hyperplane in the linear subspace $V$ and divides $V$ into two closed half-spaces. Since ${\rm Im}A\cap K\subset  {\rm bd}  K$ and ${\rm cl} N\subset V\cap K$, we deduce that ${\rm cl} N$ belongs to one the mentioned half-spaces. Clearly,  ${\rm Im}A+b$ also belongs to one of these half-spaces.   In addition, $({\rm Im}A+b)\cap {\rm cl} N\ne \emptyset$. Thus, ${\rm Im}A+b$ and ${\rm cl} N$ belong to the same closed half-space, denoted by $V^+$. Moreover, $\bar z\in{\rm cl} N\subset V^+$. Therefore, 
	$$\{t\bar z:t<0\}\cap  ({\rm Im}A+b)=\emptyset.$$ 	
	Combining this with \eqref{ray_subspace_intersection}, we get 
	\begin{equation*}\label{ray_subspace_intersection_2}\{t\bar z:t\in\mathbb{R}\}\cap  ({\rm Im}A+b)=\emptyset.
	\end{equation*} 
	This implies that $\bar z\in {\rm Im}A$, yielding $\bar z\in {\rm Im}A\cap K=\ell^+$. Consequently, ${\rm cl} N=N\cup \ell^+$.
	Hence, by observing that ${\rm cl} N\subset K$ and $\ell^+\subset {\rm bd}K$, we obtain  $$\ell^+\subset {\rm bd}_V({\rm cl} N)={\rm bd}_V N,$$ where ${\rm bd}_V D$ denotes the boundary of a set $D$ in the linear subspace $V$. Moreover, it is not difficult to see  that ${\rm bd}_V N\subset {\rm bd}K$. Additionally, since ${\rm dim}{(\rm Im}A)>1$, we infer that ${\rm dim} V>2$
	
	Next,  fixing a non-zero vector $\bar w\in \ell^+\subset {\rm bd}_VN$ and taking any $\varepsilon>0$, there exist  vectors $w^1_\varepsilon\in  B_V(\bar w,\varepsilon)\cap {\rm int}_VN$ and $w^2_\varepsilon\in  B_V(\bar x,\varepsilon)\backslash N$, where ${\rm int}_VD$ and $B_V(x,\varepsilon)$ denote the interior of $D$ and the open ball with center $x$ and radius $\varepsilon$ in the subspace $V$, respectively. Furthermore, since ${\rm dim}(B_V(\bar w,\varepsilon))={\rm dim}V>2$ and the dimension of the subspace in $\mathbb{R}^m$ generated by $\{w^1_\varepsilon\}\cup \ell^+$ is 2, we can choose $w^2_\varepsilon$ so that it does not belong to this subspace. It follows that $(w^1_\varepsilon,w^2_\varepsilon)\cap \ell^+=\emptyset$, implying $$w_\varepsilon:=(w^1_\varepsilon,w^2_\varepsilon)\cap {\rm bd}_VN\notin \ell^+.$$ In addition, $w_\varepsilon\in B_V(\bar w,\varepsilon)$. Hence,  $w_\varepsilon\in B_V(\bar w,\varepsilon)\cap ( {\rm    bd}_VN\backslash\ell^+)$. Combining this with the facts that $N\cup \ell^+={\rm cl} N=N\cup {\rm    bd}_VN$ and ${\rm bd}_V N\subset {\rm bd}K$, we get 
	$$w_\varepsilon\in B_V(\bar w,\varepsilon)\cap N\cap ( {\rm    bd}K\backslash\ell^+)\, \forall \varepsilon>0.$$ 
	The latter implies that there exists a sequence $(w^k)\subset \mathbb{R}^n$ such that 
	\begin{equation*}\label{w_k_sequence}
	(w^k)\subset	N\cap ( {\rm    bd}K   \backslash\ell^+) \text{ and } \lim_{k\to\infty}w^k=\bar w.
	\end{equation*} 
	We will show that
	\begin{equation}\label{normal_not_subset_kernel_0}
	A^*(u)=0\, \forall u\in N_K(\bar w).
	\end{equation} 
	Take any vector $ u\in N_K(\bar w)$. Since $\bar w\in {\rm bd}K$, by Proposition \ref{regular_cone}(i), we have $u\in {\rm bd}K^0$ and $\langle \bar w, u\rangle=0$.  Noting that $\{0\}\ne{\rm Im}A\cap K\subset  {\rm bd}  K$, by invoking Proposition \ref{interrior_perpendicular_cone}(ii), we obtain $\{0\}\ne{\rm Ker} A^*\cap K^\circ\subset  {\rm bd}  K^\circ$. Moreover,  $K^\circ$ is strictly convex due to Proposition \ref{smooth_cone_then_strictly_convex_1}. Hence, by applying Proposition \ref{strictly_convex_cone}(i), we deduce that ${\rm Ker} A^*\cap K^\circ$ is a ray, denoted by $\ell^-$, in $ {\rm bd}  K^\circ$.	 
	Moreover, we  have 
	\begin{equation}\label{vector_not_perpendicular_0}
	\langle \bar w,z\rangle\ne0\, \forall z\in {\rm bd}K^0\backslash\ell^-.
	\end{equation} 
	Indeed, suppose that there exists a vector $\bar z\in  {\rm bd}K^0\backslash\ell^-$ satisfying $\langle \bar w,\bar z\rangle=0$. Choosing an arbitrary vector $\tilde z\in \ell^-\backslash \{0\}$, we have $\langle \bar w,\tilde z\rangle=0$.  This yields $\langle \bar w,(\bar z+\tilde z)/2\rangle=0$. By \cite[Propostition 1.1.15]{Auslender_Teboulle_2003}, it follows that $(\bar z+\tilde z)/2\in {\rm bd}K^0$. On the other hand, since $\bar z$ and $\tilde z$ are not contained in the same ray, invoking \eqref{strictly_Convex_interior}, we derive $(\bar z+\tilde z)/2\in {\rm int}K^0$, a contradiction.
	
	Combining \eqref{vector_not_perpendicular_0} with the facts that  $u\in {\rm bd}K^0$ and $\langle \bar w, u\rangle=0$ we obtain $u\in \ell^-$, yielding $ A^* (u)=0$. Therefore, \eqref{normal_not_subset_kernel_0} is verified.
	
	On the other hand, for every $x\in  S $ satisfying $Ax+b\in  {\rm bd}   K $, we have
	\begin{equation}\label{normal_not_subset_kernel}
	A^*(u)\ne0\, \forall u\in N_K(Ax+b)\setminus\{0\}.
	\end{equation} 
	Indeed, suppose that $A^*(u)=0$ for some  some $u\in N_K(Ax+b)\setminus\{0\}$ and some $x\in  S $ satisfying $Ax+b\in  {\rm bd}   K    $. Since $b\notin {\rm Im}A$, we derive $Ax+b\ne0$. Combining this with the fact that $K$ is smooth yields $N_K(Ax+b)$ is a ray containing $u$. Hence, $A^*[N_K(Ax+b)]=\{0\}$, yielding $N_K(Ax+b)\subset {\rm Ker}A^*$. Furthermore, from Proposition \ref{regular_cone}(i) we deduce that $N_K(Ax+b)\subset K^\circ$. Thus, $N_K(Ax+b)\subset {\rm Ker} A^*\cap K^\circ.$
	From this and the fact that ${\rm Ker} A^*\cap K^\circ=\ell^-$, we get  $N_K(Ax+b)=\ell^-$. 
	
	Moreover, given any $\bar y\in \ell^-\backslash\{0\}$, applying  similar arguments as those used in proving \eqref{normal_not_subset_kernel_0}  we  obtain 
	\begin{equation}\label{vector_not_perpendicular}
	\langle \bar y,z\rangle\ne0\, \forall z\in {\rm bd}K\backslash\ell^+.
	\end{equation}

	Since $Ax+b\in {\rm bd}K\backslash\{0\}$ and  $\bar y\in N_K(Ax+b)$, by \eqref{normal_cone_form} we obtain $\langle\bar y, Ax+b\rangle=0$. Invoking \eqref{vector_not_perpendicular}, we derive $Ax+b\in \ell^+$. Combining this with the fact that $\ell^+\subset {\rm Im}A$, we get $b\in {\rm Im}A$, a contradiction. So, \eqref{normal_not_subset_kernel} holds.
	
	Assume now that \eqref{convex_cone_inclusion_linear_case} has a global error bound. Thus, there exists $\tau>0$ such that \eqref{linear_inclusion} holds.
	For each $k=1,2,\ldots$, since $w^k\in N={\rm cone}((b+{\rm Im}A)\cap K)$, there exist $t_k\ge0$ and $x^k\in S$ satisfying $w^k=t_k(Ax^k+b)$. Combining this with the fact that $w^k\in {\rm bd}K\backslash\ell^+$ yields $Ax^k+b\in {\rm bd}K\backslash\ell^+$. Hence, it follows from Proposition \ref{regular_cone}(ii) that $N_K(w^k)=N_K(Ax^k+b)$. From this and \eqref{linear_inclusion} we deduce that 
	\begin{equation}\label{global_inclusion_w_k}
	A^*\big[N_{K}(w^k)\big]\cap\bar B_{\mathbb{R}^n}\subset \tau A^*\big[N_{K}(w^k)\cap\bar B_{\mathbb{R}^m}\big]\, \forall k=1,2,\ldots.
	\end{equation}
	For each $k=1,2,\ldots$, from \eqref{normal_not_subset_kernel} we can choose $z^k\in	A^*\big[N_{K}(w^k)\big]\cap\bar B_{\mathbb{R}^n}$ such that $\|z^k\|=1$. Accordingly, by \eqref{global_inclusion_w_k}, there exists $u^k\in N_{K}(w^k)\cap\bar B_{\mathbb{R}^m}$ satisfying $z^k=\tau A^*(u^k)$. Since the sequences $(z^k)$ and $(u^k)$ are bounded, without loss of generality, we assume that $\displaystyle \lim_{k\to \infty}z^k=\bar z$ and  $\displaystyle \lim_{k\to \infty}u^k=\bar u$. The former implies that $\|\bar z\|=1$. Combining the latter with the facts that $\displaystyle \lim_{k\to \infty}w^k=\bar w$ and $u^k\in N_{K}(w^k)$ for every $k=1,2,\ldots$,  we obtain $\bar u\in N_K(\bar w)$. Thus, it follows from \eqref{normal_not_subset_kernel_0} that $A^*(\bar u)=0$. This yields
	$$\bar z=\displaystyle \lim_{k\to \infty}z^k=\lim_{k\to \infty}[\tau A^*(u^k)]=\tau A^*(\bar u)=0,$$
	contradicting the fact that $\|\bar z\|=1$. So, \eqref{convex_cone_inclusion_linear_case} has no global error bounds.

	(iii)  From the assumption that ${\rm dim} {(\rm Im}A)=1$ we get ${\rm Im}A=\ell^+\cup (-\ell^+)$. Now, we consider the following cases.

	\noindent {\bf Case 1.} $b\in {\rm Im}A$.
	
	In this case, we have 
	\begin{equation}\label{linear_inclusion_extra_1}
	Ax+b\in {\rm Im}A
	\end{equation}
	for every $x\in \mathbb{R}^n $ and 	
	\begin{equation}\label{linear_inclusion_extra_2}
	S=\{x\in\mathbb{R}^n:\,Ax+b\in K\}=\{x\in\mathbb{R}^n:\,Ax+b\in \ell^+\}. 
	\end{equation}	
	Fix any $\bar x\in S$ satisfying $A\bar x+b\ne0$.  By the continuity of $A$, there exists $\delta>0$ such that $\|(Ax+b)-(A\bar x+b)\|<\|A\bar x+b\|/2$ for every $x\in B(\bar x,\delta)$. It follows that $Ax+b\notin -\ell^+$. Combining this with \eqref{linear_inclusion_extra_1}, we get $Ax+b\in \ell^+$ for every $x\in B(\bar x,\delta)$. This, together with \eqref{linear_inclusion_extra_2}, implies that  $B(\bar x,\delta)\subset S$. Thus, we derive ${\rm int}S=\{x\in \mathbb{R}^n:\, Ax+b\in\ell^+\backslash\{0\}\} \ne\emptyset$  and ${\rm bd}S=\{x\in \mathbb{R}^n:\, Ax+b=0\} $. The former assertion means that $S$ is of full dimension. Combining this with the latter one implies that ${\rm bd}S$ is a hyperplane in $\mathbb{R}^n$. Hence, for every $x\in {\rm bd}S$, $N_S(x)$  is a ray. Moreover, if $x\in {\rm bd}S$ then $A^*[N_K(Ax+b)]=A^*(K^\circ)\ne \{0\}$, since $A^*\ne0$ and ${\rm int} K^\circ\ne\emptyset$. Therefore, $N_S(x)=A^*[N_K(Ax+b)]$ for every $x\in {\rm bd}S$. Clearly, the equation is still satisfied for every $x\in {\rm int}S$. So, the ACQ($S$) holds for \eqref{convex_cone_inclusion_linear_case}.

	For any $x\in S$,  $N_K(Ax+b)=K^\circ$ if $Ax+b=0$ and  $N_K(Ax+b)=\ell^-$ if $Ax+b\ne0$. Here, $\ell^-$ is a constant ray lying in ${\rm bd}K^0$, due to  Proposition \ref{regular_cone}(ii) and the fact that $K^\circ$ is strictly convex.  Moreover, by Theorem \ref{nonlinear_case_theorem}, \eqref{convex_cone_inclusion_linear_case} has a local error bound at every point of $ S $. Then, for some $x\in S$ satisfying $Ax+b=0$, there exists $\tau_1>0$ such that 
	$$A^*\big[N_{K}(Ax+b)\big]\cap \bar B_{\mathbb{R}^n}\subset \tau_1 A^*\big[N_{K}(Ax+b)\cap \bar B_{\mathbb{R}^m}\big],$$
	or, equivalently,
	$$	A^*(K^\circ)\cap \bar B_{\mathbb{R}^n}\subset \tau_1 A^*\big[K^\circ\cap \bar B_{\mathbb{R}^m}\big].$$
	This yields
	\begin{equation}\label{linear_inclusion_special_1}
	A^*\big[N_{K}(Ax+b)\big]\cap \bar B_{\mathbb{R}^n}\subset \tau_1 A^*\big[N_{K}(Ax+b)\cap \bar B_{\mathbb{R}^m}\big]
	\end{equation}
	for every $x\in S$ satisfying $Ax+b=0$.
	Similarly, for some $x\in S$ satisfying $Ax+b\ne0$, there exists $\tau_2>0$ such that 
	$$A^*\big[N_{K}(Ax+b)\big]\cap \bar B_{\mathbb{R}^n}\subset \tau_2 A^*\big[N_{K}(Ax+b)\cap \bar B_{\mathbb{R}^m}\big],$$
	or, equivalently,
	$$A^*(\ell^-)\cap \bar B_{\mathbb{R}^n}\subset \tau_2 A^*\big[\ell^-\cap \bar B_{\mathbb{R}^m}\big].$$
	It follows that
	\begin{equation}\label{linear_inclusion_special_2}
	A^*\big[N_{K}(Ax+b)\big]\cap \bar B_{\mathbb{R}^n}\subset \tau_2 A^*\big[N_{K}(Ax+b)\cap \bar B_{\mathbb{R}^m}\big]
	\end{equation}
	for every $x\in S$ satisfying $Ax+b\ne0$.
	Setting $\tau=\max\{\tau_1,\tau_2\}$, from \eqref{linear_inclusion_special_1} and \eqref{linear_inclusion_special_2} we deduce that \eqref{linear_inclusion} holds.
	So, the inclusion \eqref{convex_cone_inclusion_linear_case} has a global error bound.

	\noindent {\bf Case 2.} $b\notin {\rm Im}A$.	
	
	In this case, we also have that \eqref{SCQ_Affine} holds. This means that the SCQ for \eqref{convex_cone_inclusion_linear_case} is satisfied . Therefore,  ACQ($S$) holds for \eqref{convex_cone_inclusion_linear_case}. 
	
	Next,  we see that $(b+{\rm Im}A)\cap {\rm bd}K\ne\emptyset$. Indeed, if $b\in {\rm bd}K$, then the claim is true. If $b\notin K$, then taking $\bar y\in (b+{\rm Im}A)\cap {\rm int}K$, we have  $[b,\bar y]\cap {\rm bd}K\ne\emptyset$, which yields the claim.		
	Moreover, we have $N_K(Ax+b)=\{0\}$ for every $x\in S$ satisfying $(Ax+b)\in {\rm int}K$ and that   $b+{\rm Im}A$ meets ${\rm bd}K$ at most two points, as ${\rm dim} ({\rm  Im}A)=1$.  
	
	Hence, using similar arguments as in the last paragraph of Case 1, we derive that there exists $\tau>0$ such that 
	\eqref{linear_inclusion} holds. So, the inclusion \eqref{convex_cone_inclusion_linear_case} has a global error bound.
	
	The proof is complete.  $\hfill\Box$
\end{proof}

\section{Conclusions}     
This paper investigates error bounds for smooth cone-convex inclusions, providing a characterization of local error bounds using the Abadie constraint qualification. While the equivalence between local error bounds and the Abadie constraint qualification around the reference point is well-established for polyhedral cones, we extend this result to a new class of non-polyhedral convex cones that also satisfies the relationship. Additionally, we identify the conditions under which global error bounds exist for smooth cone-affine inclusions.  Applications of these findings in optimization problems involving smooth-cones, such as second-order cones, $p$-cones, and circular cones, should be conducted in the future.

\end{document}